\documentclass{article}
\usepackage{latexsym,amsfonts,amsthm,amsmath,amscd,amssymb}
\usepackage[dvips]{graphicx}

\newtheorem{lemma}{Lemma}[section]
\newtheorem{thm}[lemma]{Theorem}
\newtheorem{rem}[lemma]{Remark}
\newtheorem{prop}[lemma]{Proposition}

\newtheorem{oss}[lemma]{Observation}
\newtheorem{example}[lemma]{Example}
\newtheorem{defn}[lemma]{Definition}

\newcommand\matZ{{\mathbb{Z}}}
\newcommand\matR{{\mathbb{R}}}
\newcommand\matN{{\mathbb{N}}}

\newcommand{\cl}{C \kern -0.1em \ell}

\renewcommand{\hbar}{{\overline{h}}}

\newfont{\Got}{eufm10 scaled 1200}

\newcommand\calD{{\mathcal D}}

\begin{document}

\title{Diffeological Clifford algebras and pseudo-bundles of Clifford modules}

\author{Ekaterina~{\textsc Pervova}}

\maketitle

\begin{abstract}
\noindent We consider the diffeological version of the Clifford
algebra of a (diffeological) finite-dimensional vector space; we
start by commenting on the notion of a diffeological algebra (which
is the expected analogue of the usual one) and that of a
diffeological module (also an expected counterpart of the usual
notion). After considering the natural diffeology of the Clifford
algebra, and considering which of its standard properties re-appear
in the diffeological context (most of them), we turn to our main
interest, which is constructing the pseudo-bundles of Clifford
algebras associated to a given (finite-dimensional) diffeological
vector pseudo-bundle, and those of the usual Clifford modules (the
exterior algebras). The substantial difference that emerges with
respect to the standard context, and paves the way to various
questions that do not have standard analogues, stems from the fact
that the notion of a diffeological pseudo-bundle is very different
from the usual bundle, and this under two main respects: it may have
fibres of different dimensions, and even if it does not, its total
and base spaces frequently are not smooth, or even topological,
manifolds.

\noindent MSC (2010): 53C15, 15A69 (primary), 57R35, 57R45
(secondary).
\end{abstract}

\section*{Introduction}

The concept of the diffeological space is due to J.M. Souriau
\cite{So1}, \cite{So2}. The most comprehensive (and high-quality)
source for diffeology is the excellent book \cite{iglesiasBook},
which provides also a fascinating account of how diffeology came
about as a field. What might particularly be curious, given the
context in which Clifford algebras usually appear, is the mentioning
in the Preface of how diffeology takes sometimes a view in some
sense complementary to that of the noncommutative geometry (a brief
introduction to the latter field can be found, for instance, in
\cite{varilly}), and how it can be of help to whoever is not quite
comfortable with heavy functional analysis or
$C^*$-algebras.\footnote{That would be myself, for instance.}

\paragraph{The content} This paper is dedicated to considering Clifford
algebras for (finite-dimensional) diffeological vector spaces and
then the pseudo-bundles of such, along with pseudo-bundles of
Clifford modules. The notion of a diffeological Clifford algebra
does not bring much novelty with respect to the standard setting; it
only needs to be noticed that, unless the algebra is built over a
standard space --- which is to say, $\matR^n$ with the diffeology
made up of all the usual smooth maps, --- any smooth symmetric
bilinear form is necessarily degenerate (see \cite{iglesiasBook}),
so there is not a unitary action on the corresponding exterior
algebra (we consider some questions of which action does exist
there).

On the other hand, interesting phenomena appear when we consider,
instead of a single vector space, a collection of such that fibers
over a diffeological space --- in other words, a diffeological
counterpart of a vector bundle. Such notion appeared originally in
\cite{iglFibre} (as a partial case of that of a \emph{diffeological
fibre bundle}; see also \cite{iglesiasBook}, Chapter 8); it was also
treated in \cite{vincent}, where it is called a \emph{regular vector
bundle}, and employed in \cite{CWtangent}, where the term is a
\emph{diffeological vector space over $X$}. We opt for calling it a
\emph{diffeological vector pseudo-bundle} (although it is precisely
the same thing), both because these objects, in general, are not
bundles in the usual sense and because this avoids confusion with
vector spaces \emph{proper}.

A rough shape of such an object is a smooth surjective map $\pi:V\to
X$, where $V$ and $X$ are diffeological spaces, and the pre-image of
any point of $x$ is a finite-dimensional diffeological vector space.
Given such $\pi:V\to X$, we endow it with a smooth symmetric
bilinear form $g:X\to V^*\otimes V^*$ (as we have already commented,
in general it cannot be a metric, meaning that it does not give a
scalar product on individual fibres; we recall that it may not
always have the maximal rank possible for a given fibre). We then
turn to the subject of our main, which is pseudo-bundles of the
corresponding Clifford algebras $Cl(V,g)=\cup_{x\in
X}Cl(V_x,g(x))\to X$ and those of (abstract) Clifford modules, where
we concentrate on their behavior under the so-called \emph{gluing}
operation. Notice that all our vector spaces (in particular, the
fibres $\pi^{-1}(x)\subseteq V$) are over real numbers.

\paragraph{The structure of the paper} In order to make the paper
self-contained, we collect in Section 1 all the main definitions and
facts that are used therein, that is, diffeological spaces,
diffeological vector spaces, Clifford algebras and modules, and
diffeological algebras. In Section 2, which has somewhat expository
nature (this material shall be at least implicitly known to anyone
familiar with the diffeology field), we consider some instances of
smooth actions of diffeological algebras on diffeological spaces. We
comment on diffeological version of Clifford algebras and modules in
Section 3. After recalling, in Section 4, the needed
facts/constructions regarding diffeological vector pseudo-bundles
and pseudo-metrics on them (this material is not new and comes from
previous sources), in Sections 5 and 6 we consider the gluing
operation for pseudo-bundles of Clifford algebras (Section 5) and
those of Clifford modules (Section 6); the main result thus obtained
is that, under appropriate compatibility conditions, the result of
gluing is again a pseudo-bundle of Clifford algebras/modules in a
natural way.

\paragraph{Acknowledgments} The creation of this work is due to the
influence of a non-mathematician colleague of mine, Prof. Riccardo
Zucchi, whose good-naturedness, and the ability to provide subtle
yet eloquent and inspiring comments and pointers, are beyond any
praise. I also have a long-standing debt of gratitude to Prof. Paolo
Piazza, to whom I owe my first true encounter with the Atiyah-Singer
index theory, and in particular the first real encounter with
Clifford algebras. Despite the referee reports on the next-to-last
version of this paper being $OH$-polarized (for those who remember
some chemistry), I am grateful to both of the authors of those
reports. Finally, I am much grateful to the referees of the reports
that came in after those, for many useful suggestions.

\section{Main definitions}

We collect here the standard definitions that are needed in what
follows, with the exception of diffeological pseudo-bundles that are
given a separate treatment in the section hereafter.

\subsection{Diffeological spaces and diffeologies}

We start by defining the main concepts regarding diffeological
spaces \emph{proper} that we need.

\paragraph{Diffeological spaces and smooth maps} A diffeological
space is just a set endowed with a \emph{diffeology}, which is a
sort of a substitute for the usual notion of a smooth structure; in
and of itself, a diffeology is collection of maps into this set
satisfying several properties. The precise definition is as follows.

\begin{defn} \emph{(\cite{So2})} A \textbf{diffeological space} is a pair
$(X,\calD_X)$ where $X$ is a set and $\calD_X$ is a specified
collection of maps $U\to X$ (called \textbf{plots}) for each open
set $U$ in $\matR^n$ and for each $n\in\matN$, such that for all
open subsets $U\subseteq\matR^n$ and $V\subseteq\matR^m$ the
following three conditions are satisfied:
\begin{enumerate}
  \item (The covering condition) Every constant map $U\to X$ is a
  plot;
  \item (The smooth compatibility condition) If $U\to X$ is a plot
  and $V\to U$ is a smooth map (in the usual sense) then the
  composition $V\to U\to X$ is also a plot;
  \item (The sheaf condition) If $U=\cup_iU_i$ is an open cover and
  $U\to X$ is a set map such that each restriction $U_i\to X$ is a
  plot then the entire map $U\to X$ is a plot as well.
\end{enumerate}
\end{defn}

Usually, instead of $(X,\calD_X)$ one writes simply $X$ to denote a
diffeological space. An easy example of a diffeological space is any
smooth manifold, with the diffeology consisting of all usual smooth
maps into that manifold.

Now, if we have two diffeological spaces, $X$ and $Y$, and a set map
$f:X\to Y$ between them, then $f$ is called \textbf{smooth} if for
every plot $p:U\to X$ of $X$ the composition $f\circ p$ is a plot of
$Y$. The set of all smooth maps $X\to Y$ is denoted by
$C^{\infty}(X,Y)$.

\paragraph{The underlying topology} If $X$ is a diffeological space,
there is a natural topology underlying its diffeological structure.
It is called the \textbf{D-topology} and can be characterized as
follows: a subset $X'\subset X$ is open in D-topology (and is said
to be \textbf{D-open}) if and only if for every plot
$p:\matR^m\supset U\to X$ of $X$ the pre-image $p^{-1}(X')$ is a
usual open set in $U\subset\matR^m$. By Theorem 3.7 of
\cite{CSW_Dtopology} it actually suffices to ensure that this
condition hold for the subset of plots defined on $\matR$ (or its
subintervals); the D-topology is defined by smooth curves only.

Typically (but not always), when considering some known topological
space as a diffeological space, one would aim for a choice of
diffeology on a known (topological) space such that the underlying
D-topology coincide with the existing topology. For that, it
suffices to ensure that all plots be continuous in the usual sense.
Notice that this applies only to specific situations (say, we wish
to consider $\matR$ as a diffeological space, while preserving its
topology as a line); the set underlying the diffeological space is
not required to carry any other structure.

\paragraph{Comparing diffeologies} Given a set $X$, the set of all
possible diffeologies on $X$ is partially ordered by inclusion (with
respect to which it forms a complete lattice). More precisely, a
diffeology $\calD$ on $X$ is said to be \textbf{finer} than another
diffeology $\calD'$ if $\calD\subset\calD'$ (whereas $\calD'$ is
said to be \textbf{coarser} than $\calD$). Among all diffeologies,
there is the finest one, which turns out to be the natural
\textbf{discrete diffeology} and which consists of all locally
constant maps $U\to X$; and there is also the coarsest one, which
consists of \emph{all} possible maps $U\to X$, for all
$U\subseteq\matR^n$ and for all $n\in\matN$. It is called \emph{the}
\textbf{coarse diffeology} (or \textbf{indiscrete diffeology} by
some authors).

\paragraph{Pushforwards and pullbacks} For any diffeological space
$X$, any set $X'$, and any map $f:X\to X'$ there exists a finest
diffeology on $X'$ that makes the map $f$ smooth. It is this
diffeology that is called the \textbf{pushforward of the diffeology
of $X$ by the map $f$}; the explicit description of this diffeology
is as follows: $p':U\to X'$ is a plot if and only if for any $u_0\in
U$ there exists $u_0\in U'\subseteq U$ such that $p'|_{U'}=f\circ p$
for some plot $p$ of $X$. If, \emph{vice versa}, we have a map
$f:X'\to X$ then $X'$ can be endowed with the \textbf{pullback
diffeology}, namely, the coarsest diffeology on $X$ such that $f$ is
smooth; locally, its plots are precisely the maps $p:U\to X$ such
that $f\circ p$ is a plot of $X'$.

\paragraph{The subset diffeology} Each subset $Y\subset X$ of a
diffeological space $X$ carries a natural diffeology called
\textbf{subset diffeology} (and so is always a diffeological space,
in contrast with smooth manifolds, where most subsets are certainly
not smooth (sub)manifolds). The subset diffeology is the coarsest
diffeology such that the obvious inclusion map $Y\hookrightarrow X$
is smooth. From the practical point of view, it can be described as
the set of all those plots of $X$ whose range is contained in $Y$.

\paragraph{The quotient diffeology} Again, unlike smooth manifolds, the
diffeological spaces always have other diffeological spaces as their
quotients; these quotients are endowed a canonical diffeology,
called the \textbf{quotient diffeology}.  Specifically, if $X$ is a
diffeological space and $\sim$ is an equivalence relation on $X$,
then the quotient space $X/\sim$ is endowed with the diffeology that
is the pushforward of the diffeology of $X$ by the natural
projection $\pi:X\to X/\sim$. As is the case for all pushforward
diffeologies, locally any plot $q$ of $X/\sim$ has form $q=\pi\circ
p$ for some plot $p$ of $X$.

\paragraph{Disjoint unions and direct products} Let $\{X_i\}_{i\in I}$ be a
collection of diffeological spaces, where $I$ is a set of indices.
The \textbf{sum} of $\{X_i\}_{i\in I}$ is defined as
$$X=\coprod_{i\in I}X_i=\{(i,x)\,|\,i\in I\mbox{ and }x\in X_i\}.$$
The \textbf{sum diffeology}, or \textbf{disjoint union
diffeology},\footnote{We will mostly use the latter term, to avoid
confusion with the vector space direct sum diffeology (see below),
which is a product diffeology defined in the next sentence.} on $X$
is the \emph{finest} diffeology such that each natural injection
$X_i\to\coprod_{i\in I}X_i$ is smooth; one property of the disjoint
union diffeology is that any plot of it defined on a connected
domain, is just a plot of one the components. The \textbf{product
diffeology} $\calD$ on the product $\prod_{i\in I}X_i$ is the
\emph{coarsest} diffeology such that for each index $i\in I$ the
natural projection $\pi_i:\prod_{i\in I}X_i\to X_i$ is smooth.
Locally every plot of the product diffeology is a collection of
plots of the factors (for a finite product $X_1\times\ldots\times
X_n$ the local shape of a plot is $(p_1,\ldots,p_n)$, where $p_i$ is
a plot of $X_i$).

\paragraph{Functional diffeology} Let $X$, $Y$ be two diffeological
spaces, and let $C^{\infty}(X,Y)$ be the set of smooth maps from $X$
to $Y$. Let \textsc{ev} be the \emph{evaluation map}, defined by
$$\mbox{\textsc{ev}}:C^{\infty}(X,Y)\times X\to Y\mbox{ and }\mbox{\textsc{ev}}(f,x)=f(x).$$
The \textbf{functional diffeology} on $C^{\infty}(X,Y)$ is the
coarsest diffeology on it such that the evaluation map is smooth
(although any diffeology such that $\mbox{\textsc{ev}}$ is smooth
may be called a functional diffeology).

\subsection{Diffeological vector spaces}

The next notion which we will (obviously) make use of is that of a
diffeological vector space.

\paragraph{The definition of a diffeological vector space} Let $V$ be a
vector space over $\matR$ (this is the only case that we consider
here). A \textbf{vector space diffeology} on $V$ is any diffeology
of $V$ such that the addition and the scalar multiplication are
smooth, that is,
$$[(u,v)\mapsto u+v]\in C^{\infty}(V\times V,V)\mbox{ and }[(\lambda,v)\mapsto\lambda v]\in C^{\infty}(\matR\times V,V),$$
where $V\times V$ and $\matR\times V$ are equipped with the product
diffeology. A \textbf{diffeological vector space} is any vector
space $V$ equipped with a vector space diffeology; all of our vector
spaces will be finite-dimensional.

\paragraph{Subspaces and quotients} Every vector space subspace $W$
of a diffeological vector space is itself a diffeological vector
space for the subset diffeology; this is the diffeology with which
it is canonically endowed. Likewise, any usual quotient space $V/W$
is a diffeological vector space for the quotient diffeology; this,
again, is a canonical choice of a diffeology for it.

\paragraph{Linear maps} Even in the finite-dimensional case (or maybe
particularly so), in diffeology one needs to distinguish between
linear maps and smooth linear maps. This is because, unless the
diffeological vector space in question is a standard space,
\emph{i.e.}, $\matR^n$ with the standard diffeology, there will be
at least one non-smooth linear map from it to some other
diffeological vector space. We will see the implications of that
when we come to the definition of the diffeological dual.

\paragraph{Direct sum of diffeological vector spaces} Let
$V_1,\ldots,V_n$ be diffeological vector spaces. Consider the usual
direct sum $V=V_1\oplus\ldots\oplus V_n$ of this family; the space
$V$, equipped with the \emph{product} diffeology, is a diffeological
vector space. The diffeology on the direct sum of an infinite family
of vector spaces can be described as the coarsest diffeology such
that such that the natural projection on each factor is a smooth
linear map. Notice that the subset diffeology on each direct sum of
a finite sub-family of factors is the same as described above.

\paragraph{Euclidean structures and pseudo-metrics} The
notion of a \textbf{Euclidean diffeological vector space} does not
differ much from the usual notion of the Euclidean vector space. A
diffeological space $V$ is Euclidean if it is endowed with a scalar
product that is smooth with respect to the diffeology of $V$ and the
standard diffeology of $\matR$; that is, if there is a fixed map
$\langle , \rangle:V\times V\to\matR$ that has the usual properties
of bilinearity, symmetricity, and definite-positiveness and that is
smooth with respect to the diffeological product structure on
$V\times V$ and the standard diffeology on $\matR$. However, it is
known (see, for instance, \cite{iglesiasBook}) that a
finite-dimensional diffeological vector space admits a smooth
Euclidean structure if and only if it is diffeomorphic to the
standard $\matR^n$; in general, the maximal possible rank for a
smooth symmetric bilinear form is equal to the dimension of its
diffeological dual (see below). A form that achieves this rank
always exists and is called a \textbf{pseudo-metric}.

\paragraph{Fine diffeology on vector spaces} The \textbf{fine
diffeology} on a vector space $V$ is the \emph{finest} vector space
diffeology on it; endowed with such, $V$ is called a \emph{fine
vector space}. In the finite-dimensional case (which is the only one
we are treating here) any fine space is diffeomorphic to some
$\matR^n$, for appropriate $n$.\footnote{This is easy to see
directly; assume that a given finite-dimensional fine space $V$ is
already identified, as a vector space, with an appropriate
$\matR^n$. It suffices to see that the standard diffeology is the
finest one on $\matR^n$. Indeed, any diffeology contains all
constant maps, so for any domain $U\subset\matR^m$ (for whatever
$m$) any diffeology contains the maps $U\ni u\mapsto e_i$, where
$\{e_i\}_{i=1}^n$ is the canonical basis. Furthermore, any vector
space diffeology contains all finite linear combinations with smooth
functional coefficients of any collection of its plots. It follows
that any usual smooth map $f:U\to\matR^n$ is a plot for any vector
space diffeology on $\matR^n$, since we can write it as
$u\mapsto\sum_{i=1}^nf_i(u)e_i$. Thus, the standard diffeology is
indeed the finest vector space diffeology.}

\paragraph{The dual of a diffeological vector space} The diffeological
dual $V^*$ of a diffeological vector space $V$ (see \cite{vincent},
\cite{wu}) the set of smooth linear maps $V\to\matR$ endowed with
the functional diffeology, with respect to which it, itself, becomes
a diffeological vector space. Note that if $V$ has finite dimension
then this diffeology is standard, in the sense that $V^*$ is
diffeomorphic to some $\matR^k$ with the standard diffeology, for
appropriate $k$ (\cite{pseudometric}). This implies immediately (but
is also easy to show directly) that unless a finite-dimensional $V$
is standard itself, its diffeological dual has strictly smaller
dimension.

It should also be observed that there are many diffeological vector
spaces, with not-so-large a diffeology, that have a trivial dual. It
suffices to choose $n$ not-everywhere-smooth functions
$f_i:\matR\to\matR$ and endow $\matR^n$ with the vector space
diffeology generated by the $n$ plots $\matR\ni x\mapsto f_i(x)e_i$.
Any smooth linear $\matR$-valued map on this space must necessarily
be trivial, yet the diffeology in question can be described rather
concretely (it is essentially the extension of the ring of the usual
smooth maps to $\matR^n$ by the $n$ plots specified above).

\paragraph{The tensor product} Let $V_1,\ldots,V_n$ be finite-dimensional
diffeological vector spaces. Their usual tensor product is
canonically endowed with the \textbf{tensor product diffeology} (see
\cite{vincent} and \cite{wu}), which is the pushforward of the
product diffeology on the free product of vector spaces
$V_1\times\ldots\times V_n$ by the universal map
$V_1\times\ldots\times V_n\to V_1\otimes\ldots\otimes V_n$ (this is
slightly different from the definitions in the above sources, but
equivalent in the finite-dimensional case). The diffeological tensor
product thus defined possesses the usual universal property
(\cite{vincent}, Theorem 2.3.5): for any diffeological vector spaces
$V_1,\ldots,V_n,W$ there is a diffeomorphism between the space
$L^{\infty}(V_1\otimes\ldots\otimes V_n,W)$ of all smooth linear
maps $V_1\otimes\ldots\otimes V_n\to W$ (endowed with the functional
diffeolofy) and the space
$\mbox{Mult}^{\infty}(V_1\times\ldots\times V_n,W)$ of all smooth
multilinear maps $V_1\times\ldots\times V_n\to W$, also endowed with
the functional diffeology.

\subsection{Clifford algebras and Clifford modules}

In this section we recall briefly the standard notions of the
Clifford algebra (see, for instance, \cite{roe}) for a real vector
space equipped with a symmetric bilinear form, and then that of the
Clifford module.

\paragraph{Clifford algebras} There are a couple of ways to define a
Clifford algebra, one more abstract and one more constructive; we
recall both.

\begin{defn}
Let $V$ be a vector space equipped with a symmetric bilinear form
$q(\, ,\,)$. A \textbf{Clifford algebra} for $V$ is a unital algebra
$\cl(V,q)$ which is equipped with a map $\varphi:V\to \cl(V,q)$ such
that $\varphi(v)^2=-2q(v,v)1$, and which is universal among algebras
equipped with such maps.
\end{defn}

What the just-mentioned universality condition means precisely is
the following: if $\varphi':V\to A$ is another map from $V$ to an
algebra which satisfies $\varphi'(v)^2=-2q(v,v)1$ then there is a
unique algebra homomorphism $t:\cl(V,q)\to A$ such that
$\varphi'=t\circ\varphi$.

An easy example is the exterior algebra of a given vector space,
which corresponds to the bilinear form being identically zero.

The Clifford algebra can be equivalently defined in the following
way.

\begin{defn}
Let $V$ be a vector space equipped with a symmetric bilinear form
$q(\, ,\,)$. The \textbf{Clifford algebra} $\cl(V,q)$ associated to
$V$ and $q$ is the quotient of $T(V)/I(V)$ of the tensor algebra
$T(V)=\sum_rV^{\otimes r}$ by the ideal $I(V)\subset T(V)$ generated
by all the elements of the form $v\otimes w+w\otimes v+4q(v,w)$,
where $v,w\in V$.
\end{defn}

The universal map is then the natural projection $T(V)\to \cl(V,q)$.

\paragraph{Clifford modules} Let $V$ and $q$ be as above.

\begin{defn}
A \textbf{Clifford module} is a vector space $E$ endowed with an
action of the algebra $\cl(V,q)$, that is, a unital algebra
homomorphism $c:\cl(V,q)\to\mbox{End}(E)$.
\end{defn}

If the space $E$ is Euclidean, \emph{i.e.}, endowed with a scalar
product, one can speak of a \textbf{unitary action}, namely one for
which $c(v)$ is an orthogonal transformation for each $v\in V$.
Recall, as a main example, that the exterior algebra $\bigwedge^*V$
is a unitary Clifford module over $\cl(V,q)$.

\paragraph{Grading and filtration on $\cl(V,q)$} As is known, every
Clifford algebra $\cl(V,q)$ carries the following $\matZ_2$ grading:
$$\cl(V,q)=\cl(V,q)^0\oplus \cl(V,q)^1,$$
where $\cl(V,q)^0$ is the subspace generated by the products of an
even number of elements of $V$, while $\cl(V,q)^1$ is the subspace
generated by the products of an odd number of elements of $V$; this
is well-defined because $I(V)$ is generated by elements of even
degree in $T(V)$.

Besides, $\cl(V,q)$ inherits from $T(V)$ its filtration
$T(V)=\sum_k(\sum_{r=0}^kV^{\otimes r})$, via the natural
projection. Therefore
$$\cl(V,q)=\sum_k \cl^k(V,q),\mbox{ where }
\cl^k(V,q)=\{v\in \cl(V,q)\,|\,\exists u\in\sum_{r=0}^kV^{\otimes
r}\mbox{ such that }[u]=v\}.$$ The use of natural projections allows
also to define a surjective algebra homomorphism $$V^{\otimes k}\to
\cl^k(V,q)/\cl^{k-1}(V,q);$$ looking at its kernel, one sees that
$$\cl^k(V,q)/\cl^{k-1}(V,q)\cong\Lambda^kV.$$
This implies that the graded algebra associated to the
just-described filtration on $\cl(V,q)$, namely, the algebra
$\oplus_k \cl^k(V,q)/\cl^{k-1}(V,q)$ is isomorphic to the external
algebra $\bigwedge^* V$ (in particular, $\dim(\cl(V,q))=2^{\dim
V}$).

\subsection{Diffeological algebras and diffeological modules}

The concepts of diffeological algebra and diffeological module are
essentially obtained by adding the requirement of smoothness
wherever it is (obviously) needed; this is at least implicit in the
already existing works (see, for instance, \cite{wu}). We briefly
recall them, occasionally making some of the reasoning explicit.
Notice that an \emph{algebra} for us is always an associative unital
$\matR$-algebra.

\paragraph{Diffeological algebras} The definition of a diffeological
algebra is of course an expected one: all operations are required to
be smooth. We now state this formally.

\begin{defn}
Let $A$ be an algebra, and let $\calD$ be a diffeology on the
underlying set of $A$. The diffeology $\calD$ is called a
\textbf{diffeological algebra structure} on $A$, and $A$ is called a
\textbf{diffeological algebra}, if the pair $(A,\calD)$ is a
diffeological vector space with respect to the vector space
structure of $A$ and, furthermore, the product map $A\times A\to A$
is smooth with respect to the product diffeology on $A\times A$.
\end{defn}

Generally speaking, a vector space admits many diffeologies that
make it into a diffeological vector space (\emph{i.e.}, with respect
to which the addition and the multiplication by scalar are smooth);
it is the finest of these that is called \emph{the} diffeological
vector space structure. The same is \emph{a priori} true for
diffeological algebras, which leads us to the following definition:

\begin{defn}
Let $A$ be an algebra. The \textbf{diffeological algebra structure}
on $A$ is the finest diffeology with respect to which $A$ is a
diffeological vector space and the product in $A$ is smooth.
\end{defn}

Note that \emph{a priori} the diffeological algebra structure is
finer than just a diffeological vector space structure, even the
finest of the latter.

\paragraph{Subalgebras of a diffeological algebra} Let $(A,\calD)$
be a diffeological algebra, and let $B\subset A$ be a subalgebra of
$A$. As any subset of any diffeological space, $B$ carries a
sub-diffeology coming from $\calD$.

\begin{lemma}
The subalgebra $B$ endowed with the subset diffeology is a
diffeological algebra.
\end{lemma}

\begin{proof}
Observe first of all that $B$, being a vector subspace of the
diffeological vector space $A$, is a diffeological vector space for
the subset diffeology (see \cite{iglesiasBook}, Section 3.5).
Furthermore, if $p_1,p_2:U\to B$ are two plots of $B$ for the subset
diffeology, that are defined on the same domain, then the map $p$
defined on $U$ and acting by $p(x)=p_1(x)p_2(x)$ is a plot of $A$,
the latter being a diffeological algebra, and takes values in $B$,
the set $B$ being a subalgebra of $A$. Therefore $p$ is a plot for
the subset diffeology on $B$, which means that the product in $B$ is
smooth for the subset diffeology, whence the conclusion.
\end{proof}

\paragraph{Ideals of a diffeological algebra} We now look at the
ideals of a diffeological algebra $A$. Observe first of all that any
ideal $I$, being a vector subspace of $A$, is a diffeological vector
space with respect to the subset diffeology (as is the case for any
vector subspace of any diffeological vector space, see
\cite{iglesiasBook}, Section 3.5). The quotient $A/I$ is a
diffeological vector space for the quotient diffeology (once again,
as is the case for any diffeological vector space and any subspace
of it, see \cite{iglesiasBook}, Section 3.6). Finally, $A/I$ is an
algebra; the following easy lemma shows that it is a diffeological
one.

\begin{lemma}
Let $A$ be a diffeological algebra, and let $I$ be an ideal of $A$.
Then $A/I$ is a diffeological algebra for the quotient diffeology.
\end{lemma}

\begin{proof}
This follows from the definition of the quotient diffeology as the
pushforward of the diffeology of $A$ by the natural projection; each
operation in $A/I$ is the composition of the corresponding operation
in $A$ with the natural projection, and therefore is smooth as
(essentially) the composition of two smooth maps.
\end{proof}

\paragraph{Diffeological module} Let $A$ be a diffeological algebra,
and let $E$ be a diffeological vector space.

\begin{defn}
We say that $E$ is a \textbf{diffeological module over $A$} if there
is a fixed homomorphism
$$c:A\to L^{\infty}(E,E)$$
that is smooth for the diffeology on $A$ and the functional
diffeology on $L^{\infty}(E,E)$.
\end{defn}

Note that, as has been observed in \cite{multilinear} and, I think,
is implicit in \cite{wu}, in general there is \emph{not} a classical
isomorphism $L^{\infty}(E,E)\cong E^*\otimes E$ (where $E^*$ is the
diffeological dual of $E$, see \cite{vincent}, \cite{wu}, and
Section 2 above). Furthermore, while in the usual
(finite-dimensional) vector space context $c$ can be seen as an
element of $L(A,L(E,E))\cong A^*\otimes L(E,E)\cong A^*\otimes
E^*\otimes E$, this is not necessarily so in the diffeological
context (on the other hand, as shown in \cite{multilinear}, the
usual construction does yield a smooth map $A^*\otimes
L^{\infty}(E,E)\to L^{\infty}((A^*)^*,L^{\infty}(E,E))$, but it may
not be an isomorphism, and $(A^*)^*$ may not be diffeomorphic to
$A$).

\subsection{The exterior algebra of a diffeological vector space}

In this section we comment, not so much on the fact \emph{that} the
usual symmetrization and antisymmetrization operators (and as a
consequence, the inner product and the exterior product) are smooth
(this is expected), but on \emph{why} they are so.

\paragraph{Symmetrization and antisymmetrization operators} Let $V$
be a diffeological vector space; let $\mathcal{S}_n(V)$ stand for
the space of all symmetric $n$-tensors on $V$, and let
$\mathcal{A}_n(V)$ be the space of all antisymmetric $n$-tensors on
it. We use the usual normalized version of the the symmetrization
operator:
$$\mbox{Sym}:\underbrace{V\otimes\ldots\otimes V}_n\ni v_1\otimes\ldots\otimes v_n
\mapsto\frac{1}{n!}\sum_{\sigma\in S_n}
v_{\sigma(1)}\otimes\ldots\otimes v_{\sigma(n)}\in
\underbrace{V\otimes\ldots\otimes V}_n,$$ with $S_n$ standing for
the group of permutations on $n$ elements and the operator being
then extended by linearity.

This operator is smooth as a map $V_1\otimes\ldots\otimes V_n\to
V_1\otimes\ldots\otimes V_n$, by the properties of a vector space
diffeology and because its composition with the projection onto the
component of the sum corresponding to any fixed $\sigma\in S_n$ is
smooth. The smoothness of the latter follows easily from the
definition of the tensor product diffeology as the pushforward of
the relevant product diffeology (which is commutative in a
tautological manner --- in fact, its definition does not take the
order of factors into account) and by writing each permutation as a
composition of transpositions.

\begin{lemma}
The subset diffeology and the pushforward diffeology by $\mbox{Sym}$
on the space $\mathcal{S}_n(V)$ of all symmetric $n$-tensors on $V$
coincide.
\end{lemma}

\begin{proof}
The pushforward diffeology on $\mathcal{S}_n(V)$ is by definition
the finest such that $\mbox{Sym}$ is smooth. On the other hand, this
operator is smooth for the subset diffeology, being smooth as an
operator into $V_1\otimes\ldots\otimes V_n$, and by the very
definition of the subset diffeology. Thus, the pushforward
diffeology is \emph{a priori} finer than the subset diffeology. The
fact that they are actually the same comes from the fact that the
restriction of $\mbox{Sym}$ to $\mathcal{S}_n(V)$ is the identity.
So if $p:U\to\mathcal{S}_n(V)$ is a plot for the subset diffeology,
and $i:\mathcal{S}_n(V)\hookrightarrow V_1\otimes\ldots\otimes V_n$
is the natural inclusion, we have the equality
$p=\mbox{Sym}\circ(i\circ p)$; since $i\circ p$ is a plot of
$V_1\otimes\ldots\otimes V_n$, we get that $p$ is a plot for the
pushforward diffeology on $\mathcal{S}_n(V)$, and the statement is
proven.
\end{proof}

All the same reasoning holds for the usual antisymmetrization
operator,
$$\mbox{Alt}:\underbrace{V\otimes\ldots\otimes V}_n\ni v_1\otimes\ldots\otimes v_n
\mapsto\frac{1}{n!}\sum_{\sigma\in S_n}\mbox{sgn}(\sigma)
v_{\sigma(1)}\otimes\ldots\otimes v_{\sigma(n)}\in
\underbrace{V\otimes\ldots\otimes V}_n.$$ Thus, this operator is
smooth as well.

\begin{lemma}
The subset diffeology and the pushforward diffeology by $\mbox{Alt}$
on the space $\mathcal{A}_n(V)$ of all antisymmetric $n$-tensors on
$V$ coincide.
\end{lemma}

\begin{proof}
The proof is the same as in the case of the symmetrization operator,
since in particular the idempotent property (\emph{i.e.},
$\mbox{Alt}\circ\mbox{Alt}=\mbox{Alt}$) holds for $\mbox{Alt}$ just
the same.
\end{proof}

\paragraph{The exterior product} Let $V$ again be a
finite-dimensional diffeological space; as usual, the linear
$n$-forms on $V$ (\emph{i.e.}, all antisymmetric covariant
$n$-tensors) form the space
$$\bigwedge^n(V):=\mbox{Alt}(\underbrace{V^*\otimes\ldots\otimes V^*}_n),$$
which is also the space that previously we denoted by
$\mathcal{A}_n(V^*)$. Let us consider the usual exterior product as
a map
$$\wedge:\bigwedge^k(V)\times\bigwedge^l(V)\to\bigwedge^{k+l}(V)$$
between the diffeological vector spaces. Recall that this map acts
by the normalized antisymmetrization of the tensor product of any
two tensors, which immediately implies that it is smooth (note that
the finite-dimensional case here coincides with the standard one,
since the dual $V^*$ of a finite-dimensional diffeological $V$
always carries the standard diffeology).

\section{Instances of diffeological algebras and diffeological modules}

We consider here several instances of diffeological algebras and
diffeological modules, concentrating on providing examples of what
changes when a well-known algebra is endowed with a diffeology.

\subsection{Algebras}

Our aim here is to consider some matrix algebras, endowed with
various diffeologies, and answer a few obvious questions, such as:
what are the diffeologies such that the matrix product is \emph{not}
smooth? or the trace? or the determinant? Surprisingly or not, it is
quite easy to find matrix algebras and diffeologies where neither of
these is smooth.

\paragraph{The trace} This is probably the simplest case, meaning
that it is easy to find examples of algebras (more precisely, put,
on any algebra of square matrices, diffeologies) such that the trace
is not a smooth function. A simple example of this sort readily
comes by extending our examples of non-standard diffeological vector
spaces.

\begin{example}\label{algebra:non:diffeological:ex}
Let $V=\mathcal{M}_{2\times 2}(\matR)$ be the usual algebra of
$2\times 2$ matrices with real coefficients; endow it with the
vector space diffeology generated by the plot $p:\matR\ni
x\mapsto\left(\begin{array}{cc} 0 & 0 \\ 0 & |x|\end{array}\right)$.
Then it is obvious that the trace is not a smooth function (into the
standard $\matR$), since the composition $\mbox{tr}\circ p$, which
is the map $x\mapsto|x|$, is not a smooth function.
\end{example}

The example just given is a rather crude attempt, since, as we will
discover shortly, $V$ is not a diffeological algebra (it is easy to
see that the matrix product is not smooth). Let us briefly see the
case of the determinant.

\paragraph{The determinant} It is also easy to see that the
determinant is not a smooth function on $V$ of the above example.
Indeed, since the diffeology in question is a vector space
diffeology, the map $q:\matR\ni x\mapsto
p(x)+\left(\begin{array}{cc} 1 & 0 \\ 0 & 0\end{array}\right)$ is a
plot of it (being the sum of $p$ with the constant plot
$x\mapsto\left(\begin{array}{cc} 1 & 0 \\ 0 & 0\end{array}\right)$).
Once again, $\det\circ q$ is the map $x\mapsto|x|$, so the
determinant is not smooth as a map into the standard $\matR$.

\paragraph{The product} Finally, it is equally easy to see that the
product map on $V$ is not a smooth map in $V$. It suffices to
consider the product $\left(\begin{array}{cc} 1 & 1\\
0 & 1\end{array}\right)\left(\begin{array}{cc} 0 & 0\\
0 & |x|\end{array}\right)=\left(\begin{array}{cc} 0 & |x|\\
0 & |x|\end{array}\right)$. The matrix on the right is not a plot of
$V$ (for any plot $p':U\to V$ of $V$ the $(1,2)$th coefficient of
$p(u)$ is an ordinary smooth function in $u$), so $V$ is not a
diffeological algebra.

\paragraph{Generating an algebra diffeology} Always in reference to
our above example, the question of what is the analogue in case of
algebras of the vector space diffeology generated by a given
(collection of) plot(s), is rather natural. Of course, the abstract
answer is obvious and is (at least) implicit elsewhere: if $A$ is an
algebra and $\mathcal{A}$ is a collection of maps from domains of
Euclidean spaces to $A$ then the \textbf{algebra diffeology
generated by $\mathcal{A}$} is the finest diffeology on $A$ that
contains $\mathcal{A}$ and such that the addition, scalar
multiplication, and the algebra product are all smooth. What we are
wondering about is a concrete description, something which would
give us an idea of, for instance, the algebra diffeology on the
above $V$ generated by $p$.

\begin{oss}
The algebra diffeology on $V$ generated by $p$ is the vector space
diffeology generated by the following four maps:
$$p_{11}:\matR\ni x\mapsto\left(\begin{array}{cc} |x| & 0\\ 0 & 0\end{array}\right),\,\,\,
p_{12}:\matR\ni x\mapsto\left(\begin{array}{cc} 0 & |x|\\
0 & 0\end{array}\right),$$
$$p_{21}:\matR\ni x\mapsto\left(\begin{array}{cc} 0 & 0\\ |x| & 0 \end{array}\right),\,\,\,
p_{22}:\matR\ni x\mapsto\left(\begin{array}{cc} 0 & 0\\
0 & |x|\end{array}\right).$$
\end{oss}

\begin{proof}
We need to show two things: one, that the vector space diffeology
generated by the maps $p_{ij}$ is an algebra diffeology, and, two,
that any algebra diffeology on $V$ that contains $p=p_{22}$, also
contains $p_{11}$, $p_{12}$, and $p_{21}$. Now, the reason why the
vector space diffeology generated by these four maps is also an
algebra diffeology is that any plot of it either is locally constant
or locally filters through a plot of form
$$\matR\ni x\mapsto\left(\begin{array}{cc}
f_{11}(x)+g_{11}(x)|h_{11}(x)| & f_{12}(x)+g_{12}(x)|h_{12}(x)|\\
f_{21}(x)+g_{21}(x)|h_{21}(x)| & f_{22}(x)+g_{22}(x)|h_{22}(x)|
\end{array}\right),$$
where $f_{ij},g_{ij},h_{ij}$ are some usual smooth functions. Now,
since $|x|^2=x^2$ is a smooth function, the extension of the ring
$C^{\infty}(\matR,\matR)$ by the absolute value function
$x\mapsto|x|$ is of degree $1$ over $C^{\infty}(\matR,\matR)$, which
implies precisely that the product of two matrices as above is again
a matrix of the same form, that is, the product is smooth as a map
$V\times V\to V$.

Finally, in order to prove that an algebra diffeology on $V$
containing $p_{22}$, must also contain all the other maps $p_{ij}$,
it suffices to consider the products
$$\left(\begin{array}{cc} 0 & 1\\ 0 & 0\end{array}\right)\cdot\left(\begin{array}{cc} 0 & 0\\ 0 & |x|\end{array}\right)=
\left(\begin{array}{cc}0 & |x|\\ 0 & 0\end{array}\right),\,\,\,
\left(\begin{array}{cc} 0 & 0\\
0 & |x|
\end{array}\right)\cdot\left(\begin{array}{cc} 0 & 0\\ 1 & 0\end{array}\right)=
\left(\begin{array}{cc} 0 & 0\\ |x| & 0\end{array}\right),$$
$$\left(\begin{array}{cc} 0 & 1\\ 0 & 0\end{array}\right)\cdot\left(\begin{array}{cc} 0 & 0\\ 0 & |x|\end{array}\right)
\cdot\left(\begin{array}{cc} 0 & 0\\
1 & 0\end{array}\right)=\left(\begin{array}{cc} |x| & 0\\
0 & 0\end{array}\right).
$$ (The matrices on the right-hand side of these three expressions must be values at
$x$ of some plot of the diffeological algebra $V$, and these values
are respectively $p_{12}(x)$, $p_{21}(x)$, and $p_{11}(x)$).
\end{proof}

\paragraph{Are $\det$ and $\mbox{tr}$ always smooth for a diffeological
algebra?} Suppose, on the other hand, that an algebra of matrices
$\mathcal{M}_{n\times n}(\matR)$ is endowed with an algebra
diffeology. It is then natural to wonder whether the determinant and
the trace are necessarily smooth functions. The answer is negative
and easily follows from the example already considered, that of
$V=\mathcal{M}_{2\times 2}(\matR)$ endowed with the algebra
diffeology generated by the plot $p:x\mapsto\left(\begin{array}{cc} 0 & 0\\
0 & |x|\end{array}\right)$.

\begin{lemma}
If $V$ is as above, $\det:V\to\matR$ and $\mbox{tr}:V\to\matR$ are
not smooth as functions into the standard $\matR$.
\end{lemma}

\begin{proof}
The trace is not smooth since $\mbox{tr}\circ p$ is the absolute
value function, \emph{i.e.}, not a plot for the standard $\matR$. As
for the determinant, note first that, since any constant map is a
plot for any diffeology, and the addition is smooth, the map $q:x\mapsto\left(\begin{array}{cc} 1 & 0\\
0 & |x|\end{array}\right)$ is a plot of $V$. Since $\det\circ q$ is
again the absolute value function, it is not a plot of the standard
$\matR$, which implies that $\det$ is not a smooth function.
\end{proof}

\subsection{Modules}

We now turn to illustrate in a concrete fashion the concept of a
diffeological module. We first comment on the usual action of the
matrix algebra $\mathcal{M}_{n\times n}(\matR)$ on a
finite-dimensional diffeological vector space $V$, identified with
$\matR^n$ and endowed with some (non-standard) diffeology $\calD_V$.

\subsubsection{The functional diffeology on $L^{\infty}(V,V)$ and
multiplication by matrices}

Let us briefly explain the relation between the two. By the standard
property of the functional diffeology, a map $t:U\to
L^{\infty}(V,V)$ is a plot for the functional diffeology if (and
only if) for every plot $q:U'\to V$ of $V$ the evaluation of $t(u)$
on $s(u')$ is a plot of $V$, with the domain of definition $U\times
U'$. So if $\mathcal{M}_{n\times n}(\matR)$ is endowed with an
algebra diffeology $\calD_{\mathcal{M}}$, its action on $V$ by the
usual left multiplication is smooth if and only if for every plot
$A:U\to\mathcal{M}_{n\times n}(\matR)$ of the matrix algebra and for
every plot $p:U'\to V$ of $V$ the assignment
$$U\times U'\ni(u,u')\mapsto A(u)\cdot p(u')\in V$$
defines again a plot of $V$. The condition, of course, remains the
same if in place of $\mathcal{M}_{n\times n}(\matR)$ we consider any
of its subalgebras (with an algebra diffeology).

\subsubsection{The action of $\mathcal{M}_{2\times 2}(\matR)$ on a
non-standard $\matR^2$}

This is one of the simplest examples of a diffeological module,
which however indicates already the main differences of the concept
upon introducing a diffeological structure.

\paragraph{The domain of action and the subalgebra with smooth action}
Let $V$ be a diffeological vector space, $V=\matR^2$ with the vector
space diffeology generated by the plot $p:x\mapsto(0,|x|)$, and let
$\mathcal{A}=\mathcal{M}_{2\times 2}(\matR)$ the algebra of $2\times
2$ matrices, whose diffeology we do not specify for the moment.
Indeed, we are about to discover that $\mathcal{A}$ taken as a whole
does not admit a smooth action on $V$, and this is because it
contains matrices that while they obviously induce homomorphisms
$V\to V$, these homomorphisms are not
smooth. As an example, it suffices to take the matrix $\left(\begin{array}{cc} 1 & 1\\
0 & 1\end{array}\right)$; in general, it is easy to observe that a
matrix $A\in\mathcal{A}$ induces a smooth endomorphism of $V$ if and
only if $(0,1)$ is an eigenvector of it. This means only the
lower-triangular matrices (\emph{i.e.}, matrices $A$ such that
$(A)_{12}=0$) act smoothly on $V$. Denote by $\mathcal{A}_{tr}$ the
algebra of all such matrices.

\paragraph{The diffeology on the subalgebra acting} Let us now consider
the potential diffeologies on $\mathcal{A}_{tr}$. The ultimate
condition that such diffeologies should satisfy is that the
action-by-left-multiplication map
$$c:\mathcal{A}_{tr}\to L^{\infty}(V,V)$$
be smooth for the chosen diffeology on $\mathcal{A}_{tr}$ and the
functional diffeology on $L^{\infty}(V,V)$; and this, by the
characterization of functional diffeologies, means that, for every
plot $q:U\to\mathcal{A}_{tr}$ of $\mathcal{A}_{tr}$ and for every
plot $s:U'\to V$ of $V$, the map $U\times U'\ni(u,u')\mapsto
q(u)\cdot s(u')\in V$ must be a plot of $V$.

\begin{prop}
Let $\mathcal{A}_{tr}$ be endowed with a diffeology $\calD_{matr}$
such that the left-multiplication action $c$ on $V$ is smooth. Then
$\calD_{matr}$ is the standard diffeology on $\mathcal{A}_{tr}$.
\end{prop}

\begin{proof}
Let us consider a pair $(q,s)$, where $q:U\to\mathcal{A}_{tr}$ is a
plot of $\mathcal{A}_{tr}$ and $s:U\to V$ is a plot $V$; let us see
under which conditions the map $(u,u')\mapsto q(u)\cdot s(u')$ is a
plot of $V$. Now, $s(u')=(s_1(u'),s_2(u'))$, where $s_1:U'\to\matR$
is a usual smooth function, $s_1\in C^{\infty}(U',\matR)$, while
$s_2$ belongs to the extension of the ring $C^{\infty}(U',\matR)$ by
the absolute value function. In practice, this means that
$s_2(u')=f(u')+g(u')|h(u')|$ for some $f,g,h\in
C^{\infty}(U',\matR)$. Note that the product $q(u)\cdot s(u')$ must
again be of this form (with $U\times U'$ in place of $U'$).

Let us write $q(u)=\left(\begin{array}{cc} q_{11}(u) & 0 \\
q_{21}(u) & q_{22}(u)\end{array}\right)$; we need to show that the
three functions $q_{ij}$ are ordinary smooth functions. Considering
first the case of $s$ the constant plot with the value $(1,0)$ (and
using the characterization of plots of $V$ just given), we
immediately conclude that $q_{11}$ and $q_{21}$ must be ordinary
smooth functions, while taking the constant plot with value $(0,1)$
we conclude that $q_{22}$ must have the form indicated for $s_2$,
\emph{i.e.}, $q_{22}(u)=f_1(u)+g_1(u)|h_1(u)|$ for some
$f_1,g_1,h_1\in C^{\infty}(U,\matR)$.

Suppose that $q_{22}$ is not a smooth function; then it has form
$f_1(u)+g_1(u)|h_1(u)|$, where $g_1$ and $h_1$ are not everywhere
zero. However, the diffeology in question is an algebra diffeology,
so the following must also be a plot of it:
$$u\mapsto \left(\begin{array}{cc}
q_{11}(u) & 0 \\ q_{21}(u) &
q_{22}(u)\end{array}\right)\cdot\left(\begin{array}{cc} 1
& 0\\ 1 & 1\end{array}\right)=\left(\begin{array}{cc} q_{11}(u) & 0 \\
q_{21}(u)+q_{22}(u) & q_{22}(u)\end{array}\right).$$ Since we have
already established that the $(2,1)$-th component of any plot of
$\mathcal{A}_{tr}$ must be a smooth function, we obtain a
contradiction with the assumption that $q_{22}$ is not smooth. The
statement is proven.
\end{proof}

\paragraph{A subalgebra with smooth action and non-standard
diffeology} There does however exist a subalgebra of
$\mathcal{M}_{2\times 2}(\matR)$ with smooth action on $V$ and
non-standard diffeology. This is the subalgebra of diagonal
matrices, endowed with the algebra diffeology generated by the plot
$p(x)=\left(\begin{array}{cc} 0 & 0\\ 0 & |x|\end{array}\right)$
(this statement is quite obvious, so we make no further comment).

\subsubsection{The case of $\mathcal{M}_{3\times 3}(\matR)$ and
non-standard $\matR^3$}

Let us now consider a larger example. Specifically, we take
$V=\matR^3$ endowed with the vector space diffeology generated by
the plot $p:\matR\ni x\to |x|e_3\in V$; as in the case of dimension
$2$, consider $\mathcal{A}=\mathcal{M}_{3\times 3}(\matR)$, the
algebra of $(3\times 3)$-matrices, that acts as usual, by left
multiplication.

\paragraph{Smooth action on $V$} We can generalize a number of
observations already made in the case of $n=2$, namely, a matrix
$A\in\mathcal{A}$ defines a smooth endomorphism of $V$ if and only
if $(0,0,1)$ is an eigenvector of it. Since this implies that
$(A)_{13}=(A)_{23}=0$, we must restrict our attention to the
subalgebra $\mathcal{A}'$ of matrices satisfying this condition.

\paragraph{Diffeology on $\mathcal{A}'$} Let us consider the
possible choice(s) of algebra diffeology on $\mathcal{A}'$ that make
its action on $V$ smooth. Let $p:U\ni u\mapsto\left(\begin{array}{ccc} p_{11}(u) & p_{12}(u) & 0\\
p_{21}(u) & p_{22}(u) & 0\\ p_{31}(u) & p_{32}(u) & p_{33}(u)
\end{array}\right)$ be a plot of an algebra diffeology on
$\mathcal{A}'$. Taking products with the vector $(1,0,0)^t$ and
$(0,1,0)^t$ allows us to see that
$p_{11},p_{12},p_{21},p_{22}:U\to\matR$ must be ordinary smooth
functions for the action to be smooth, while taking the product with
$(0,0,1)^t$ shows that $p_{31},p_{32},p_{33}$ are of form
$f(u)+g(u)|h(u)|$ for some smooth functions $f,g,h$. It remains to
take into account the fact that the diffeology on $\mathcal{A}'$ is
an algebra diffeology, that is, that for any two plots $p,q$ of the
form just indicated their product is again a plot. This amounts to
checking that $(p(u)q(u))_{11}$, $(p(u)q(u))_{12}$,
$(p(u)q(u))_{21}$, $(p(u)q(u))_{22}$ are in $C^{\infty}(U,\matR)$,
while $(p(u)q(u))_{31}$, $(p(u)q(u))_{32}$, $(p(u)q(u))_{33}$ belong
to the extension of $C^{\infty}(U,\matR)$ by the absolute value
function; and this is established by the obvious calculation.

\paragraph{The block subalgebra $\mathcal{A}_{bl}$ and its action}
By $\mathcal{A}_{bl}$ we denote the subalgebra of $\mathcal{A}'$
that satisfies the additional condition that $(A)_{31}=(A)_{32}=0$.
It is quite easy to see that the coarsest algebra diffeology for
which the action of $\mathcal{A}_{bl}$ on $V$ is smooth consists
precisely of plots of form
$$p:U\ni u\mapsto\left(\begin{array}{ccc} p_{11}(u) & p_{12}(u) & 0\\
p_{21}(u) & p_{22}(u) & 0\\ 0 & 0 & p_{33}(u) \end{array}\right),$$
where $p_{11},p_{12},p_{21},p_{22}:U\to\matR$ are ordinary smooth
functions, while $p_{33}$ has form $u\mapsto f(u)+g(u)|h(u)|$, again
for ordinary smooth functions $f,g,h$. Finally, it is also easy to
see that the only other algebra diffeology on $\mathcal{A}_{bl}$
such that the left-multiplication map is smooth is the standard
diffeology (more precisely, the subset diffeology relative to the
standard diffeology on the standard algebra $\mathcal{M}_{3\times
3}(\matR)$).

\subsubsection{Other considerations}

Given a finite-dimensional diffeological space $V$, there is the
unique subspace $V_0$ of $V$ that is maximal for the following two
properties:
\begin{enumerate}
  \item the subset diffeology on $V_0$ is the standard diffeology;
  \item there is another subspace $V_1$ of $V$ such that $V=V_0\oplus
  V_1$ as a diffeological vector space, that is, the diffeology of
  $V$ is the direct sum diffeology relative to the subset diffeology
  on $V_0$ and $V_1$.
\end{enumerate}
The dimension of $V_0$ is that of the diffeological dual $V^*$ (see
\cite{pseudometric}).

Choose now a basis $\{v_1,\ldots,v_k\}$ of $V_0$ and a basis
$\{v_{k+1},\ldots,v_n\}$ of $V_1$. With respect to the basis
$\{v_1,\ldots,v_n\}$ of $V$ thus obtained, a matrix
$A\in\mathcal{M}_{n\times n}(\matR)$ defines a smooth endomorphism
of $V$ only if we have:
$$A=\left(\begin{array}{cc} A_k & 0\\ B & C \end{array}\right),$$
where $A_k\in\mathcal{M}_{k\times k}(\matR)$ and $0$ is the zero
$(k\times(n-k))$-matrix (then $B$ is of course of size $(n-k)\times
k$ and $C$ is of size $(n-k)\times(n-k)$). Note that it is only a
necessary condition, not a sufficient one.

Denoting by $\mathcal{A}_{n,k}$ the algebra of all $(n\times
n)$-matrices of this form (\emph{i.e.}, such that the upper
right-hand block of size $k\times(n-k)$ is the zero matrix), we
observe that (for the same reasons as it occurs in our examples) for
any plot $p:U\to\mathcal{A}_{n,k}$ of an algebra diffeology on
$\mathcal{A}_{n,k}$ for which the left-multiplication action on $V$
is smooth, the functions $p_{ij}:U\to\matR$ given by taking the
$(i,j)$-th coefficient of the matrix $p(u)$ are smooth in the usual
sense for $i,j=1,\ldots,k$. Not much can be said in general about
the rest of the functions $p_{ij}$, as it much depends on the
specific diffeology of $V$ (more precisely, on its non-standard
part).

\section{Diffeological Clifford algebras}

We now turn to the diffeological version of the Clifford algebra,
associated to a finite-dimensional diffeological vector space.

\subsection{The definition of the diffeology}

The explicit construction of each Clifford algebra $\cl(V,q)$ as a
quotient of the tensor algebra $T(V)$ yields the obvious quotient
diffeology on $\cl(V,q)$ (which is the pushforward of the diffeology
of $T(V)$ by the natural projection). This should be considered as
the most natural choice for the diffeology on $\cl(V,q)$ (the
possible alternatives are rather artificial, so we do not discuss
them).

\paragraph{The quotient diffeology from the tensor algebra} Let us
outline the construction of this diffeology. Recall first of all
that each space $V^{\otimes k}$, with $k=1,\ldots,n=\dim(V)$,
carries the tensor product diffeology (defined in \cite{vincent} and
\cite{wu}, and recalled in Section 2). Now, the tensor algebra
$T(V)$ is the direct sum of these spaces,
$T(V)=\bigoplus_{k=1}^nV^{\otimes k}$; we endow it with the sum
diffeology. Denote the resulting diffeology on $T(V)$ by
$\calD_{\otimes}$.

\begin{lemma}
The tensor algebra $T(V)$ considered with the diffeology
$\calD_{\otimes}$ is a diffeological algebra.
\end{lemma}

\begin{proof}
A direct sum of diffeological vector spaces is always a
diffeological vector space for the product diffeology, so we only
need to explain why the tensor product operation is smooth as a map
$$V^{\otimes k}\times V^{\otimes l}\to T(V).$$
This follows automatically from the definition of the tensor product
diffeology, which (for us) is the quotient diffeology corresponding
to the product diffeology. Indeed, the product in $T(V)$ can be seen
as the composition
$$V^{\otimes k}\times V^{\otimes l}\to V^{\otimes(k+l)}\hookrightarrow T(V)$$
of the quotient projection with the natural inclusion, both of which
are smooth by construction.\footnote{Note that the natural embedding
of each summand into a direct sum of vector spaces being smooth is a
property of the product diffeology coming from vector space
diffeologies. There is no analogue of it for a general product
diffeology; then again, this is more because for a generic direct
product there are no canonical inclusions of the factors.}
\end{proof}

Since the Clifford algebra over $V$ is the quotient of $T(V)$ by its
ideal, the quotient diffeology coming from $\calD_{\otimes}$ makes
$\cl(V,q)$ into a diffeological algebra; we will consider $\cl(V,q)$
with this diffeology.

\paragraph{The universal map for diffeological Clifford algebras}
The diffeological analogue of the universal property for Clifford
algebras is the following statement.

\begin{prop}
Let $V$ be a diffeological vector space endowed with a smooth
bilinear symmetric form $q(,)$, let $A$ be a diffeological algebra
with unity, and let $f:V\to A$ be a smooth linear map such that
$$f(v)\cdot f(w)+f(w)\cdot f(v)=-4q(v,w)\cdot 1_A\mbox{ for all }v,w\in V.$$
Then there exists a (unique) smooth algebra homomorphism
$\hat{f}:\cl(V,q)\to A$ that extends $f$.
\end{prop}

\begin{proof}
What we essentially need to do is to check whether the standard
universal homomorphism $f^{\otimes}:T(V)\to A$ extending $f$ is
smooth. Indeed, by assumption on $f$, the homomorphism $f^{\otimes}$
yields a well-defined $\cl(V,q)\to A$, which, by definition of the
quotient diffeology, is smooth if and only if $f^{\otimes}$ is
smooth. To conclude, observe that this latter property follows from
the smoothness of the product in $A$; as an illustration, let us
show, for instance, that the standard extension $\tilde{f}:V\otimes
V\to A$ (given by $\tilde{f}(v\otimes w)=f(v)f(w)$) is smooth.

Indeed, let $p:U\to V\otimes V$ be a plot; recall that locally $p$
writes as $p=\pi\circ(p_1,p_2)$, where $p_1,p_2$ are plots of $V$
defined (up to passing to a suitably small sub-domain) on $U$ and
$\pi:V\times V\to V\otimes V$ is the natural projection. Then we
have $(\tilde{f}\circ p)(u)=p_1(u)p_2(u)$, which is locally a plot
of $A$, since the multiplication in $A$ is smooth.
\end{proof}

\subsection{The exterior algebra as a diffeological Clifford module}

It has been recalled above that for each Clifford algebra $\cl(V,q)$
the exterior algebra $\bigwedge V$ is a Clifford module over
$\cl(V,q)$. We now show that this is true in the diffeological
context as well; more precisely, we show that the standard action of
$\cl(V,q)$ on $\bigwedge V$ is smooth. For the duration of this
section we assume that $q$ is a (diffeological) scalar product.

\paragraph{The standard action $c$} The usual action
$c:\cl(V,q)\to\mbox{End}(\bigwedge V)$ is obtained by first defining
it on $V$ and then extending it to the rest of $\cl(V,q)$; on $V$,
this action is given by setting
$$c(v)=\varepsilon(v)-i(v)\mbox{ for all }v\in V,$$
where $\varepsilon(v)$ acts on $\bigwedge V$ by the left exterior
product, $\varepsilon(v)(\alpha)=v\wedge\alpha$, for all
$\alpha\in\bigwedge V$, and $i(v)$ is the adjoint of
$\varepsilon(v)$ via $q$, \emph{i.e.}, it acts by
$$i(v)(w_1\wedge\ldots\wedge w_l)=\sum_{i=1}^l(-1)^{i+1}w_1\wedge\ldots\wedge q(v,w_i)\wedge\ldots\wedge w_l.$$

\paragraph{The action $c$ is smooth} This essentially follows from
the smoothness of all operations and that of $q$ (observe that, for
instance, the smoothness of the exterior product is due to the
diffeology of $\bigwedge V$ being the quotient diffeology of that of
$T(V)$ and equivalently, the pushforward of the latter by the
antisymmetrization operator).

\paragraph{The isomorphism $\cl(V,q)\cong\bigwedge V$} Recall that (in the
case when $q$ is a scalar product) there is a standard isomorphism
$\sigma:\cl(V,q)\to\bigwedge V$ acting by
$$\sigma(x)=c(x)(1)\mbox{ for every }x\in\cl(V,q).$$
We have the following:

\begin{prop}
The isomorphism $\sigma$ is smooth.
\end{prop}

\begin{proof}
Let $p:U\to\cl(V,q)$ be a plot of the Clifford algebra $\cl(V,q)$;
we need to show that $\sigma\circ p$ is a plot of $\bigwedge V$. As
we have seen above, the action $c$ is smooth, which means that
$c\circ p:U\to L^{\infty}(\bigwedge V,\bigwedge V)$ is a plot for
the functional diffeology on $L^{\infty}(\bigwedge V,\bigwedge V)$.
By 1.57 of \cite{iglesiasBook}, this is equivalent to the smoothness
of the following map: $C:U\times\bigwedge V\to\bigwedge V$ acting by
$C(u,\alpha)=c(p(u))(\alpha)$. The map $C$ being smooth means that
its composition with any plot of $U\times\bigwedge V$ is a plot of
$\bigwedge V$. Observe now that the map $p':U\to U\times\bigwedge V$
given by $p'(u)=(u,1)$ is a plot of $U\times\bigwedge V$, and that
$C\circ p'=\sigma\circ p$, whence the conclusion.
\end{proof}

\subsection{Smooth decompositions of diffeological vector
spaces}\label{smooth:deco:vector:sp:sect}

Before going further, we need to verify a few statements regarding
some aspects of the behavior of diffeology with respect to
pushforwards; we do so for diffeological vector spaces, this being a
more general case. The following definition explains the matter.

\begin{defn}
Let $V$ be a diffeological vector space, and let $V=V_0\oplus V_1$
be its decomposition into a direct sum of two subspaces. We say that
this decomposition is \textbf{smooth} if the direct sum diffeology
relative to the subset diffeologies on $V_0$ and $V_1$ coincides
with the diffeology of $V$.
\end{defn}

Note that \emph{a priori} the direct sum diffeology is finer than
$V$'s own diffeology (that this can easily happen is illustrated by
Example 2.5 in \cite{pseudometric}). What we are now interested in
(although we do not consider this point comprehensively), is what
happens with (non)smoothness of a decomposition under linear maps.

\begin{lemma}
Let $V$ be a diffeological vector space, let $W$ be a vector space,
and let $f:V\to W$ be a surjective linear map. Let $V'\leqslant V$
be a subspace of $V$, and let $W'=f(V')$ be the corresponding
subspace of $W$. Endow $W$ with the pushforward of the diffeology of
$V$ by the map $f$; then the corresponding subset diffeology on $W'$
is in general coarser than the pushforward of the subset diffeology
on $V'$ by the map $f|_{V'}$. If $\mbox{Ker}(f)$ splits off as a
smooth direct summand then the two diffeologies coincide.
\end{lemma}

\begin{proof}
Let $\calD_s$ be the subset diffeology on $W'$, and let
$\calD_{psh}$ be the pushforward diffeology on it. By its
definition, $\calD_{psh}$ consists of plots of local form
$f|_{V'}\circ p'=f\circ p'$, where $p'$ is a plot for the subset
diffeology on $V'$, \emph{i.e.} it is a plot of $V$ with range in
$V'$. Thus, $f\circ p'$ is a plot for the pushforward diffeology on
$W$, which by assumption takes values in $W'$ only; therefore it is
automatically a plot for the subset diffeology on $W$. We conclude
that $\calD_{psh}\subseteq\calD_s$.

Now let $q:U\to W'$ be a plot of $\calD_s$; this means that it is a
plot of $W$ of taking values in $W'$ only. And so it is a plot for
the pushforward diffeology on $W$, \emph{i.e.} that locally (on each
open subset $U'\subset U$ small enough) it has form $q=f\circ p$ for
some plot $p:U'\to V$. If the range of $p$ is contained in $V'$ then
$p$ is a plot for the subset diffeology on $V'$, and we can conclude
that $q$ is a plot for the pushforward of the subset diffeology on
$V'$ by $f|_{V'}$. So assume that the range of $p$ is not contained
in $V'$.

If $\mbox{Ker}(f)$ splits off as a smooth direct summand,
$V=V_1\oplus\mbox{Ker}(f)$ then (on a small enough subset $U'\subset
U$) the plot $p$ can be represented as $p=p_1\oplus p_{ker}$ where
both $p_1$ and $p_{ker}$ are plots of $V$, with ranges contained in
$V_1$ and $\mbox{Ker}(f)$ respectively. Observe furthermore that the
equality $q=f\circ p$ implies that $p(U')\subset V'+\mbox{Ker}(f)$;
and this implies, in turn, that the range of $p_1$ is actually
contained in $V'$. Thus, we actually have $q=f\circ p_1$, where
$p_1$ is a plot for the subset diffeology on $V'$, so $q$ is a plot
of $\calD_{psh}$.
\end{proof}

\begin{rem}
That the kernel of $f$ splits off as a smooth direct summand is a
sufficient condition, but not a necessary one.
\end{rem}

We now consider the implications of the above lemma on the behavior
of smooth decompositions under projections (essentially, under
taking quotients).

\begin{lemma}\label{when:smooth:deco:descends:lem}
Let $V$ be a diffeological vector space, let $W$ be a vector space,
and let $f:V\to W$ be a surjective linear map; endow $W$ with the
pushforward diffeology. Let $V=V_0\oplus V_1$ be a smooth
decomposition of $V$ into a direct sum, and suppose that $f(V_0)\cap
f(V_1)=\{0\}$. Then the decomposition $W=W_0\oplus W_1$, where
$W_i=f(V_i)$ for $i=0,1$, is smooth.
\end{lemma}

\begin{proof}
The decomposition $V=V_0\oplus V_1$ being smooth means that every
plot $p$ of $V$ locally has form $p_0+p_1$ where $p_i$ takes values
in $V_i$ for $i=0,1$ (which is equivalent to saying that $p_i$ is a
plot for the subset diffeology on $V_i$). Let us now consider an
arbitrary plot $q$ of $W$. By definition of the pushforward
diffeology, locally it has form $f\circ p$ for some plot $p$ of $V$
(to account for the constant plots, recall that $f$ is surjective).
By linearity of $f$ and the above observation on the structure of
plots of $V$, we obtain that $q=f\circ p_0+f\circ p_1$, where
$q_i:=f\circ p_i$ takes values in $W_i$ for $i=0,1$ and, $p_i$ being
a plot of $V$, is also a plot of $W$. The plot $q$ being arbitrary,
we obtain that $W=W_0\oplus W_1$ is a smooth decomposition.
\end{proof}

\subsection{The grading of a diffeological Clifford algebra}

Our aim here is to show that the standard grading of a diffeological
Clifford algebra is in fact a smooth decomposition of it as a
diffeological vector space. It obviously corresponds to the standard
grading of a usual Clifford algebra:
$$\cl(V,q)=\cl(V,q)^0\oplus\cl(V,q)^1.$$

\begin{prop}
The grading of any diffeological Clifford algebra $\cl(V,q)$ is a
smooth direct sum decomposition.
\end{prop}

\begin{proof}
What we have to show is that the direct sum (in the sense of vector
spaces) diffeology on $\cl(V,q)^0\oplus\cl(V,q)^1$ relative to the
subset diffeologies on $\cl(V,q)^0$ and $\cl(V,q)^1$ coincides with
the diffeology on $\cl(V,q)$ (\emph{a priori} the direct sum
diffeology is finer). It suffices to show that each plot $p$ of
$\cl(V,q)$ locally writes as a sum $p=p_0+p_1$, where each $p_i$ is
a plot of the subset diffeology on $\cl(V,q)^i$ for $i=0,1$.

Recall that each plot of $\cl(V,q)$ is locally a composition of a
plot of $T(V)$ with its universal projection on $\cl(V,q)$. Let
$\pi:T(V)\to\cl(V,q)$ be this projection, let
$T^0(V)=\bigoplus_rV^{\otimes 2r}$ be the subspace of tensors of
even degree, and let $T^1(V)=\bigoplus_rV^{\otimes(2r+1)}$ be the
subspace of tensors of odd degree. This decomposition is smooth by
the definition of the diffeology on $T(V)$. Furthermore, by the
standard reasoning, $\pi(T^0(V))=\cl(V,q)^0$ and
$\pi(T^1(V))=\cl(V,q)^1$. Then by Lemma
\ref{when:smooth:deco:descends:lem}, we obtain that
$\cl(V,q)=\cl(V,q)^0\oplus\cl(V,q)^1$ is a smooth decomposition for
the pushforward of the diffeology of $T(V)$ by the map $\pi$, that
is, the diffeology of $\cl(V,q)$ as defined.
\end{proof}

\subsection{The filtration of a diffeological Clifford algebra}

Let us now consider the standard filtration of $\cl(V,q)$ by
$\cl^k(V,q)/\cl^{k-1}(V,q)$, with $\cl^k(V,q)$ defined, as usual, as
the image of the restriction $\pi^k$ to
$T^k(V):=\bigoplus_{r=0}^kV^{\otimes r}$ of the natural projection
$\pi:T(V)\to\cl(V,q)$.

\begin{lemma}
The subset diffeology on $\cl^k(V,q)$ relative to its inclusion
$\cl^k(V,q)\subset\cl(V,q)$ coincides with the pushforward of the
(subset) diffeology on $T^k(V)\subset T(V)$ by the map $\pi^k$.
\end{lemma}

\begin{proof}
This is a direct consequence of definitions involved. The subset
diffeology on $\cl^k(V,q)$ consists of all plots of $\cl(V,q)$ with
range contained in $\cl^k(V,q)$. The diffeology of $\cl(V,q)$ is the
pushforward of the diffeology of $T(V)$ by the map $\pi$, so any
plot of it has form $\pi\circ p$ for a plot $p$ of $T(V)$. By the
definition of the map $\pi$, the range of $\pi\circ p$ is contained
in $\cl^k(V,q)$ if and only if the range of $p$ is contained in
$T^k(V)$, on which we obviously have $\pi=\pi^k$, whence the
conclusion.
\end{proof}

This lemma has the following (expected) implication.

\begin{prop}
The usual isomorphism $\cl^k(V,q)/\cl^{k-1}(V,q)\to\bigwedge^kV$ is
a smooth map with smooth inverse (for the quotient diffeology on
$\cl^k(V,q)/\cl^{k-1}(V,q)$).
\end{prop}

\section{Diffeological vector pseudo-bundles and pseudo-metrics}

As we have just seen, an individual diffeological Clifford algebra,
or Clifford module, represents simply a particular instance of the
standard construction; what is more interesting is considering a
collection of them, a diffeological counterpart of a vector bundle
of Clifford algebras/modules. However, unless we are dealing with a
standard case, it is frequently not a true bundle; instead, it is a
diffeological counterpart of such, a more complicated object that
was first introduced in \cite{iglFibre} (see also
\cite{iglesiasBook}), and then considered, for example, in
\cite{vincent} and then in \cite{CWtangent}. The terminology was
different in all these sources (it was just \emph{diffeological
fibre bundle} for \cite{iglFibre}, \emph{regular vector bundle} in
\cite{vincent}, and \emph{diffeological vector space over $X$} in
\cite{CWtangent}); we call it a \emph{diffeological vector
pseudo-bundle}, finding this term more convenient for stressing its
difference from usual bundles (as well as avoiding confusion with
individual vector spaces).

\subsection{Diffeological vector pseudo-bundles}

The main difference between a diffeological vector pseudo-bundle and
a usual smooth vector bundle is the absence of local
trivializations. Indeed, diffeological pseudo-bundles are frequently
not locally trivial, or they are not bundles at all, in the sense
that the fibre may easily not be the same over different points. An
easy example of such is the wedge of a line and a plane,
\emph{i.e.}, of $\matR$ and $\matR^2$, at the origin, which fibers
over the wedge of two copies of $\matR$ (one of which is the
covering $\matR$ and the other is the $x$-axis of the $\matR^2$).
The fibres are of two sorts; in one case, they are points
(considered as trivial vector spaces, and in the other case they are
lines of form $\{x\}\times\matR$, with the obvious vector space
structure on the second factor). It is easy to endow the two spaces,
$\matR\vee_0\matR^2$ and $\matR\vee_0\matR$, with diffeologies such
that the above-described projection
$\matR\vee_0\matR^2\to\matR\vee_0\matR$ of one onto the other be
smooth, and the subset diffeology on the pre-image of each point be
a vector space diffeology. This means that we obtain what is called
a diffeological vector pseudo-bundle, which below we formally
define.

\paragraph{The formal definition} Let $V$ and $X$ be two diffeological
spaces, and let $\pi:V\to X$ be a smooth surjective map.

\begin{defn}
The map $\pi:V\to X$ is called a \textbf{diffeological vector
pseudo-bundle}, if for every $x\in X$ the pre-image $\pi^{-1}(x)$ is
endowed with the structure of a vector space such that the following
three conditions are satisfied:
\begin{enumerate}
\item the map $V\times_X V\to V$ induced by the addition within each
fibre is smooth for the subset diffeology on $V\times_X V$ relative
to its inclusion $V\times_X V\subset V\times V$ and the product
diffeology on the latter;

\item the map $\matR\times V\to V$ induced by the fibrewise scalar
multiplication is smooth for the product diffeology on $\matR\times
V$ relative to the standard diffeology on $\matR$;

\item the zero section is smooth as a map $X\to V$.
\end{enumerate}
\end{defn}

We will also use the term \emph{pseudo-bundle} to refer to just the
total space $V$. As already mentioned, the requirement of the
existence of an atlas of local trivializations is absent; in
particular, having fibres of different dimensions is explicitly
admitted. An example of this has already been described in the
beginning of this section, and we add some more below. Other
instances appear in the literature, also arising from independent
contexts, such as the internal tangent bundle to the union of
coordinate axes in $\matR^2$ constructed in \cite{CWtangent}.

\paragraph{Operations on vector pseudo-bundles} The usual operations on
smooth vector bundles (direct sum, tensor product, and so on) can
still be defined for diffeological vector pseudo-bundles, although
due to the absence of local trivializations this is done with a
different approach. It appeared first in \cite{vincent}, Chapter 5
(some details are added in \cite{pseudobundles}).
\begin{itemize}
\item \textbf{Sub-bundles and quotient pseudo-bundles}. Let $\pi:V\to
X$ be a finite-dimensional diffeological vector pseudo-bundle, and
let $W\subset V$ be such that for every $x\in X$ the intersection
$W\cap\pi^{-1}(x)$ is a vector subspace of $\pi^{-1}(x)$. Then
$\pi|_W:W\to X$, where $W$ is endowed with the subset diffeology
relative to $V$, is also a diffeological vector pseudo-bundle,
called a \textbf{sub-bundle of $V$} (in general, it is also a
pseudo-bundle). \emph{Vice versa}, any collection
$$\bigcup_{x\in X}W_x\,\,\,\mbox{ such that }\,\,\,W_x\leqslant\pi^{-1}(x)$$
of vector subspaces, one for each fibre, is a sub-bundle of $V$ for
the subset diffeology and the restriction of $\pi$ onto it.

Let now $W\subset V$ be a sub-bundle of $\pi:V\to X$. Then $W$
defines an obvious equivalence relation on $V$, by taking the
quotient of each fibre $\pi^{-1}(x)$ over the subspace
$W\cap\pi^{-1}(x)$. The quotient space $V/W$,
$$V/W=\bigcup_{x\in X}\left(\pi^{-1}(x)/(W\cap\pi^{-1}(x))\right),$$
endowed with the quotient diffeology and the obvious projection onto
$X$, becomes a vector pseudo-bundle itself and is called the
\textbf{quotient pseudo-bundle}. Obviously, the subset diffeology on
each $\pi^{-1}(x)/\left(W\cap\pi^{-1}(x)\right)\subset V/W$
coincides with its quotient diffeology relative the subset
diffeology on $\pi^{-1}(x)$.

\item \textbf{Direct sum of two pseudo-bundles}. Let $\pi_1:V_1\to
X$ and $\pi_2:V_2\to X$ be two diffeological vector pseudo-bundles
with the same base space. Their \textbf{direct sum}, as a set, is
$$V_1\oplus V_2:=V_1\times_X
V_2=\{(v_1,v_2)\,|\,\pi_1(v_1)=\pi_2(v_2),v_i\in V_i,i=1,2\}\subset
V_1\times V_2;$$ it is endowed with the obvious factorwise
operations of addition and scalar multiplication and with the subset
diffeology relative to the product diffeology on $V_1\times V_2$, as
well as with the obvious projection $\pi_1\oplus\pi_2$ onto $X$. It
is a matter of technicality then (see, for instance, \cite{vincent})
that $\pi_1\oplus\pi_2:V_1\oplus V_2\to X$ is a diffeological vector
pseudo-bundle with the fibre
$$(\pi_1\oplus\pi_2)^{-1}(x)=\pi_1^{-1}(x)\oplus\pi_2^{-1}(x)$$ (as
\emph{diffeological} vector spaces) for every $x\in X$.

\item \textbf{Tensor product of two pseudo-bundles}. Let, again,
$\pi_1:V_1\to X$ and $\pi_2:V_2\to X$ be two diffeological vector
pseudo-bundles. Their \textbf{tensor product} is defined as the
quotient pseudo-bundle of the direct product product bundle
$V_1\times_X V_2$ over its sub-bundle $W$ which is defined by
requiring each fibre $W_x$ of $W$ to be precisely the kernel of the
universal projection
$\pi_1^{-1}(x)\times\pi_2^{-1}(x)\to\pi_1^{-1}(x)\otimes\pi_2^{-1}(x)$.
By construction, each fibre of this quotient has form
$\pi_1^{-1}(x)\otimes\pi_2^{-1}(x)$.

\item \textbf{The dual pseudo-bundle}. Let $\pi:V\to X$ be a
diffeological vector pseudo-bundle, and let $$V^*=\cup_{x\in
X}(\pi^{-1}(x))^*$$ be the union of the diffeological duals of all
its fibres. Each dual obtains its diffeological vector space
structure as the dual of the diffeological vector space
$\pi^{-1}(x)\subset V$. The set $V^*$ is then endowed with the
obvious projection onto $X$ and with the smallest diffeology
relative to which all fibres recover their original diffeology as
the subset diffeology.

That such a diffeology exists, and the following characterization of
its plots of the dual bundle diffeology, comes from \cite{vincent},
Definition 5.3.1 and Proposition 5.3.2. Let $U$ be a domain of some
$\matR^l$. A map $p:U\to V^*$ is a plot for the dual bundle
diffeology on $V^*$ if and only if for every plot $q:U'\to V$ the
map
$$Y'\ni(u,u')\mapsto p(u)(q(u'))\mbox{ on }
Y'=\{(u,u')\,|\,\pi^*(p(u))=\pi(q(u'))\in X\}\subset U\times U'$$ is
smooth for the subset diffeology of $Y'\subset\matR^{l+\dim(U')}$
and the standard diffeology of $\matR$.

\item \textbf{The pseudo-bundle of linear maps}. Let $\pi_1:V_1\to X$ and $\pi_2:V_2\to X$ be two
finite-dimensional diffeological vector pseudo-bundles over the same
base space $X$. For each $x\in X$ consider the space
$L^{\infty}((V_1)_x,(V_2)_x)$ of smooth linear maps
$(V_1)_x\to(V_2)_x$ endowed with the functional diffeology; this is
a (finite-dimensional) diffeological vector space. The union
$$\mathcal{L}(V_1,V_2):=\bigcup_{x\in X}L^{\infty}((V_1)_x,(V_2)_x),$$
endowed with the obvious projection $\pi_L$ onto $X$, is a
diffeological vector pseudo-bundle, that carries the following
diffeology, called the \textbf{pseudo-bundle functional diffeology}.
It is the smallest diffeology generated by all the maps
$p:U\to\mathcal{L}(V_1,V_2)$, with $U$ being a domain of some
$\matR^m$, that possess the following property: for every plot
$q:\matR^{m'}\supset U'\to V_1$ the map
$$(u,u')\mapsto p(u)(q(u'))\mbox{ defined on }
Y'=\{(u,u')\,|\,\pi_L(p(u))=\pi_1(q(u'))\}\subset U\times U'$$ (and
taking values in $V_2$) is smooth for the subset diffeology of
$Y'\subset\matR^{m+m'}$ and the diffeology of $V_2$.
\end{itemize}

\paragraph{Pseudo-bundle diffeology generated by some plots}
Let $X$ be a diffeological space, and let $\pi:V\to X$ be a
surjective map, where $V$ is a set such that for every $x\in X$ the
pre-image $\pi^{-1}(x)$ carries a vector space structure.

\begin{defn}
Let $\mathcal{A}=\{p_i:U_i\to V\,|\,i\in I\}$ be a collection of
maps such that $\pi\circ p_i$ is a plot of $X$ for all $i\in I$. The
\textbf{pseudo-bundle diffeology on $V$ generated by $\mathcal{A}$}
is the smallest diffeology on $V$ that contains $\mathcal{A}$ and
with respect to which $\pi:V\to X$ becomes a diffeological vector
pseudo-bundle.
\end{defn}

The existence of such a diffeology can be easily inferred from
Proposition 4.16 of \cite{CWtangent}. Note that the notion of the
pseudo-bundle diffeology generated by a set of given plots is
distinct from that of the generated diffeology \emph{proper} and
that of the vector space diffeology generated by a plot.

\paragraph{Examples} Obviously, usual vector bundles satisfy
the definition of a diffeological vector pseudo-bundle (as is
formally established in \cite{pseudometric-pseudobundle}, they
actually serve as instances of a kind of diffeological assembly,
called \emph{gluing}, which is described in the next section). Other
simple examples include all projections
$\matR^k\times\matR^m\to\matR^k$, where $\matR^k$ is endowed with
any diffeology, $\matR^m$ with any vector space diffeology, and the
total space, with the product diffeology. We will also give an
example of a non-trivial non-manifold pseudo-bundle when it comes to
discussing the so-called \emph{pseudo-metrics} (substitutes of
Riemannian metrics) on diffeological vector pseudo-bundles.

\subsection{More complicated pseudo-bundles: diffeological gluing}

Diffeological gluing can be thought of as a partial way to
compensate for the absence of local trivializations in the concept
of a diffeological vector pseudo-bundles, although it is not aimed
to describe all possible pseudo-bundles (and likely does not). The
idea is, just like many usual smooth vector bundles can be seen as
the result of identification, along the transition functions, of
several copies of the product bundle $\matR^{n+k}\to\matR^n$, so
many diffeological pseudo-bundles can be obtained as the result of
diffeological gluing of several trivial (or almost trivial, although
this case we will avoid) pseudo-bundles of form
$\matR^{n_i+k_i}\to\matR^{n_i}$. There are two main differences: one
is that these trivial pseudo-bundles may have varying dimensions
(and/or diffeologies), and the gluing may not be along
diffeomorphisms, nor is it required, for the subsets being
identified, to have non-empty interior. Still, the latter case is a
generalization of the former (the rigorous statement and the proof
can be found in \cite{pseudometric-pseudobundle}).

\subsubsection{Gluing of diffeological spaces and of diffeological
vector pseudo-bundles}

Here we define formally the operation of gluing and the resulting
diffeology, first for individual diffeological spaces and then for
diffeological vector pseudo-bundles (the latter actually amounts to
the former performed twice). The main aspect is that of the choice
of diffeology, since from the topological/set-theoretic point of
view this construction is well-known.

\paragraph{The gluing of two diffeological spaces} Underlying this
operation is the standard definition of gluing of two topological
spaces. Although this can be defined in a more general context, we
limit ourselves to the case of gluing along maps that are invertible
(the assumption that they are smooth is always present). Then, as is
standard, if $X_1,X_2$ are two topological spaces and $f:X_1\supset
Y_1\to Y_2\subset X_2$ is a homeomorphism then the result of gluing
$X_1$ and $X_2$ along $f$ is the space $X=(X_1\sqcup X_2)/\sim$,
where $\sim$ is the equivalence relation given by $x_1\sim x_1$ for
$x_1\in X_1\setminus Y_1$, $x_2\sim x_2$ for $x_2\in X_2\setminus
Y_2$, and $Y_1\ni y_1=f(y_1)\in Y_2$. We will denote the space $X$
thus obtained by $X_1\cup_f X_2$:
$$X_1\cup_f X_2:=(X_1\sqcup X_2)/\sim.$$

\begin{defn}
Let $X_1$ and $X_2$ be diffeological spaces, and let $f:X_1\supseteq
Y_1\to Y_2\subseteq X_2$ be a map smooth for the subset diffeologies
on $Y_1$ and $Y_2$. The \textbf{gluing diffeology} on the space
$X_1\cup_f X_2$ is the quotient diffeology corresponding to the
disjoint union diffeology on $X_1\sqcup X_2$; endowed with this
diffeology, $X_1\cup_f X_2$ is said to be the result of
\textbf{diffeological gluing} of $X_1$ to $X_2$ along $f$.
\end{defn}

\paragraph{The gluing of two diffeological vector pseudo-bundles}
This is a direct sequence of the previous construction, under
appropriate assumptions on the gluing maps. Specifically, let
$$\pi_1:V_1\to X_1\,\,\mbox{ and }\,\,\pi_2:V_2\to X_2$$ be two
(finite-dimensional) diffeological vector pseudo-bundles, let
$$f:X_1\supset Y_1\to Y_2\subset X_2$$ be a diffeomorphism for the
subset diffeologies on $Y_1$ and $Y_2$, and let
$$\tilde{f}:\pi_1^{-1}(Y_1)\to\pi_2^{-1}(Y_2)$$ be a smooth map that
is linear on each fibre and is a lift of $f$, namely, $\pi_1\circ
f=\tilde{f}\circ\pi_2$. This latter condition ensures that between
the diffeological spaces $V_1\cup_{\tilde{f}}V_2$ and $X_1\cup_f
X_2$ there is a well-defined (smooth) map
$$\pi_1\cup_{(\tilde{f},f)}\pi_2:V_1\cup_{\tilde{f}}V_2\to X_1\cup_f
X_2 $$ induced by $\pi_1$ and $\pi_2$, and it is easy to check (see
\cite{pseudobundles}) that $\pi_1\cup_{(\tilde{f},f)}\pi_2$ is again
a diffeological vector pseudo-bundle.

\subsubsection{Gluing and other operations}

Recall that in the case of the usual smooth vector bundles, the
operations on them are defined via their restrictions to the charts
of a given atlas of local trivializations. What happens in the
diffeological context is a procedure that in some sense is an
inverse of this standard strategy: the operations are defined on the
whole pseudo-bundles, and then one wonders what, if any, relation
there is with our diffeological substitute for such atlas, which is
a representation as the result of gluing of some finite number of
trivial diffeological pseudo-bundles.

What we find, under this perspective is that in a number of cases
the two commute. In particular, if we are also given a pseudo-bundle
endowed with a so-called pseudo-metric (see the next section), this
allows us to reduce the constructions of the corresponding
pseudo-bundle of Clifford algebras and that of the exterior algebras
(seen as the pseudo-bundle of Clifford modules; see the last
section), essentially to the case of standard trivial bundles, or at
least diffeologically trivial bundles.

\paragraph{The commutativity problem} The question of whether
the gluing operation on pseudo-bundles commutes with the usual
operations on vector bundles (such as the direct sum, tensor
product, and taking the dual bundle) is therefore a very natural
one. Below we summarize the answer to it (which is respectively yes,
yes, no for the three main operations just listed, see
\cite{pseudobundles} and \cite{pseudometric-pseudobundle} for more
details).

To fix the notation, consider two pairs of pseudo-bundles with their
respective gluings. One pair is $\pi_1:V_1\to X_1$ and $\pi_2:V_2\to
X_2$ that are glued along the maps $f:X_1\supset Y_1\to Y_2\subset
X_2$ (a diffeomorphism for the subset diffeologies on $Y_1$ and
$Y_2$) and $\tilde{f}:\pi_1^{-1}(Y_1)\to\pi_2^{-1}(Y_2)$ (smooth and
linear on each fibre), yielding the pseudo-bundle
$$\pi_1\cup_{(\tilde{f},f)}\pi_2:V_1\cup_{\tilde{f}}V_2\to X_1\cup_f X_2.$$
Let the other pair be $\pi_1':V_1'\to X_1$ and $\pi_2':V_2'\to X_2$;
these are glued along the same map $f$ on the bases and a given lift
$\tilde{f}':(\pi_1')^{-1}(Y_1)\to(\pi_2')^{-1}(Y_2)$ of it, yielding
the pseudo-bundle
$$\pi_1'\cup_{(\tilde{f},f)}\pi_2':V_1'\cup_{\tilde{f}'}V_2'\to
X_1\cup_f X_2.$$ The comparisons to make would be between the
following pseudo-bundles.

\paragraph{Direct sum} Consider the following direct sum
pseudo-bundles:
$$\pi_1\oplus\pi_1':V_1\oplus V_1'\to X_1\mbox{ and }
\pi_2\oplus\pi_2':V_2\oplus V_2'\to X_2.$$ These can be glued along
the maps $f$ and its lift
$$\tilde{f}\oplus\tilde{f}':\pi_1^{-1}(Y_1)\oplus(\pi_1')^{-1}(Y_1)\to
\pi_2^{-1}(Y_2)\oplus(\pi_2')^{-1}(Y_2).$$ It turns out that the
result of this gluing, the pseudo-bundle
$$(\pi_1\oplus\pi_1')\cup_{(\tilde{f}\oplus\tilde{f}',f)}(\pi_2\oplus\pi_2'):
(V_1\oplus V_1')\cup_{\tilde{f}\oplus\tilde{f}'}(V_2\oplus V_2')\to
X_1\cup_f X_2,$$ is canonically diffeomorphic, with diffeomorphism
that is a lift of the identity diffeomorphism of the bases, to the
pseudo-bundle
$$(\pi_1\cup_{(\tilde{f},f)}\pi_2)\oplus(\pi_1'\cup_{(\tilde{f}',f)}\pi_2'):
(V_1\cup_{\tilde{f}}V_2)\oplus(V_1'\cup_{\tilde{f}'}V_2')\to
X_1\cup_f X_2.$$ In other words, \emph{gluing commutes with the
direct sum}.

\paragraph{Tensor product} The analogous question for the tensor product
takes the following form. We state right away that between the two
possible (pre-gluing) tensor products $V_1\otimes V_1'$ and
$V_2\otimes V_2'$ there is the naturally (fibrewise) defined map
$$\tilde{f}\otimes\tilde{f}':\pi_1^{-1}(Y_1)\otimes(\pi_1')^{-1}(Y_1)\to
\pi_2^{-1}(Y_2)\otimes(\pi_2')^{-1}(Y_2),$$ which is smooth, linear
on each fibre, and a lift of $f$. This allows to form the glued
pseudo-bundle
$$(\pi_1\otimes\pi_1')\cup_{(\tilde{f}\otimes\tilde{f}',f)}(\pi_2\otimes\pi_2'):
(V_1\otimes V_1')\cup_{\tilde{f}\otimes\tilde{f}'}(V_2\otimes
V_2')\to X_1\cup_f X_2.$$ The question is whether it is
diffeomorphic to the tensor product pseudo-bundle
$$(\pi_1\cup_{(\tilde{f},f)}\pi_2)\otimes(\pi_1'\cup_{(\tilde{f}',f)}\pi_2'):
(V_1\cup_{\tilde{f}}V_2)\otimes(V_1'\cup_{\tilde{f}'}V_2')\to
X_1\cup_f X_2.$$ Also in this case the answer is positive.

\paragraph{Dual pseudo-bundles} Assume that $f$ is smoothly invertible.
We now wish to compare the following two pseudo-bundles:
$$(\pi_1\cup_{(\tilde{f},f)}\pi_2)^*:(V_1\cup_{\tilde{f}}V_2)^*\to X_1\cup_f X_2\mbox{ and }
\pi_2^*\cup_{(\tilde{f}^*,f^{-1})}\pi_1^*:V_2^*\cup_{\tilde{f}^*}V_1^*\to
X_2\cup_{f^{-1}}X_1.$$ In general, they are not diffeomorphic: since
over the domain of gluing the resulting fibre is always the target
space of the lifted part of the pair of maps that define the gluing,
we must also have, for each $y_1\in Y_1$ (equivalently, for each
$y_2\in Y_2$) that $(\pi_1^{-1}(y_1))^*$ is diffeomorphic, as a
diffeological vector space, to $(\pi_2^{-1}(f(y_1)))^*$.

Note also that, even if the equality
$(\pi_1^{-1}(y_1))^*=(\pi_2^{-1}(f(y_1)))^*$ holds for all $y\in Y$,
it is not sufficient to ensure that
$(\pi_1\cup_{(\tilde{f},f)}\pi_2)^*=\pi_2^*\cup_{(\tilde{f}^*,f^{-1})}\pi_1^*$.
It suffices to think of the open annulus and the open M\"obius
strip, both considered with their usual smooth structure and their
usual fibering over a circle. These smooth structures are a
particular case of a diffeological structure, of course, and the
projection onto the circle (also with its standard diffeology) is a
particular case of a diffeological vector pseudo-bundle for each of
them. Obviously, these two spaces are not even homeomorphic, let
alone diffeomorphic.

Thus, another necessary condition for the gluing-dual commutativity
is that the dual map $\tilde{f}^*$ be a diffeomorphism of its domain
with its image. It is shown in
\cite{exterior-algebras-pseudobundles} that this condition becomes
sufficient under the additional assumption that the factors of
gluing admit the so-called \emph{compatible pseudo-metrics} (see the
next section).

\subsection{Pseudo-metrics on pseudo-bundles}

It is known (see \cite{iglesiasBook}) that a finite-dimensional
diffeological vector space in general does not admit a scalar
product, therefore the straightforward extension of the notion of
the Riemannian metric to diffeological vector pseudo-bundles does
not make much sense. Indeed, such an extension would be a smooth
section $g:X\to V^*\otimes V^*$ of a given finite-dimensional
diffeological vector pseudo-bundle $\pi:V\to X$ with each value
$g(x)$ a smooth scalar product on the fibre; thus, on most
pseudo-bundles it would not exist at all. It is therefore replaced
by the notion of a \textbf{pseudo-metric on a pseudo-bundle}, which
is obtained by replacing scalar products on fibres with symmetric
bilinear forms of maximal rank. Such an object does not always
exists either, but it is applicable to a much wider class of
pseudo-bundles.

\subsubsection{What is a pseudo-metric on a pseudo-bundle?}

We now describe an obvious extension of the notion of a
pseudo-metric from the case of a finite-dimensional diffeological
vector space to that of a diffeological vector pseudo-bundle.

\paragraph{The notion} Let $\pi:V\to X$ be a diffeological vector
pseudo-bundle.

\begin{defn}
A \textbf{pseudo-metric} on the pseudo-bundle $\pi:V\to X$ is any
smooth section $g$ of the pseudo-bundle
$\pi^*\otimes\pi^*:V^*\otimes V^*\to X$ such that for every $x\in X$
the form $g(x)$ is a pseudo-metric on the diffeological vector space
$\pi^{-1}(x)$, that is, it is semi-definite positive symmetric of
rank $\dim((\pi^{-1}(x))^*)$.
\end{defn}

Note that by Theorem 2.3.5 of \cite{vincent} $g(x)$ is indeed a
bilinear map on $\pi^{-1}(x)$, so the definition is well-posed. We
have already announced that it does not, however, always exist (we
give an example below).

\paragraph{Pseudo-bundles that do not admit pseudo-metrics} For the
definition as stated, it is quite easy to find examples of
pseudo-bundles that do not admit any pseudo-metric. Namely, what
happens (as in the case of the example below) is, there is a smooth
section $X\to V^*\otimes V^*$ that satisfies all the conditions to
be a pseudo-metric except one: it does not have the maximal rank
possible everywhere, but only outside of some positive codimension
subset. (It is quite similar to an example that appears in
\cite{pseudobundles}).

\begin{example}\label{no:pseudometric:exists:ex}
Let $V=\matR^2$ and $X=\matR$; let $\pi:V\to X$ be the projection on
the $x$-axis, $\pi(x,y)=x$. Endow $X$ with the standard diffeology,
and endow $V$ with the pseudo-bundle diffeology generated by the
plot $p:\matR^2\to V$ defined by $p(u,v)=(u,u|v|)$.
\end{example}

The following was proven in \cite{pseudometric-pseudobundle}. The
main point is that the rank of the prospective pseudo-metric on $V$
would have to be zero everywhere except over the origin, where it
should be $1$; despite the diffeology being non-standard, it is
impossible to find a smooth section satisfying this condition.

\begin{lemma}
If $\pi:V\to X$ is the pseudo-bundle of Example
\ref{no:pseudometric:exists:ex}, it does not admit a pseudo-metric.
\end{lemma}

The basic reason for this is that the evaluation of any prospective
pseudo-metric on a constant non-zero section is proportional to the
$\delta$-function, with non-zero coefficient.

\subsubsection{Pseudo-bundles obtained by gluing and pseudo-metrics}

We have already indicated the pseudo-bundles obtained by gluing of
some (a finite number of) standard (diffeologically trivial) bundles
as our main object of interest.\footnote{Admittedly, I would not
know what to do with a more general case.} Since our final goal is
to consider the corresponding pseudo-bundles of Clifford algebras
(and then those of Clifford modules), we need to clarify how
pseudo-metrics behave with respect to the gluing; indeed, this
behavior is not uniquely defined.

\paragraph{The compatibility condition} This condition is in fact
immediately obvious (out of geometric considerations) as a necessary
one for the existence of some kind of induced pseudo-metric. Let
$\pi_1:V_1\to X_1$ and $\pi_2:V_2\to X_2$ be two pseudo-bundles,
with a gluing between them, given by a smooth invertible
$f:X_1\supset Y\to f(Y)\subset X_2$ and its smooth fibrewise linear
lift $\tilde{f}:\pi_1^{-1}(Y)\to\pi_2^{-1}(f(Y))$. Suppose that each
pseudo-bundle $V_i$ carries a pseudo-metric $g_i$.

\begin{defn}
We say that $g_1$ and $g_2$ are \textbf{compatible} if for every
$y\in Y$ and for all $v_1,v_2\in\pi_1^{-1}(y)$ we have
$$g_1(y)(v_1,v_2)=g_2(f(y))(\tilde{f}(v_1),\tilde{f}(v_2)).$$
\end{defn}

\paragraph{Inductions into glued spaces} Note (see
\cite{pseudometric-pseudobundle}) that the following obvious
inclusions
$$i_1:X_1\setminus Y\hookrightarrow X_1\cup_f X_2,\,\,\,
i_2:X_2\hookrightarrow X_1\cup_f X_2,$$
$$j_1:V_1\setminus
\pi_1^{-1}(Y)\hookrightarrow V_1\cup_{\tilde{f}}V_2,\,\,\,
j_2:V_2\hookrightarrow V_1\cup_{\tilde{f}}V_2$$ are all
inductions.\footnote{Recall that a map $h:X\to Y$ between two
diffeological spaces is called an \emph{induction} if it is
injective, and the subset diffeology on $h(X)\subseteq Y$ coincides
with the pushforward of the diffeology of $X$ by the map $h$.}
Furthermore, the images of $i_1$ and $i_2$ are disjoint and cover
$X_1\cup_f X_2$ (the same holds for $j_1$, $j_2$, and
$V_1\cup_{\tilde{f}}V_2$). We will frequently make use of this fact
and the notation (and the respective counterparts of these in other
cases) when defining various maps involving glued
pseudo-bundles/spaces.

\paragraph{The induced pseudo-metric} The following
statement, proven in \cite{pseudometric-pseudobundle}, shows that a
pseudo-bundle obtained by gluing carries a natural pseudo-metric, as
long as the factors of gluing are endowed with compatible ones.

\begin{thm}
Let $\pi_1:V_1\to X_1$ and $\pi_2:V_2\to X_2$ be two
finite-dimensional diffeological vector pseudo-bundles, let
$(\tilde{f},f)$ be a gluing between them, and let $g_1$ and $g_2$ be
pseudo-metrics on $V_1$ and, respectively, $V_2$ compatible with
respect to the gluing. Define $\tilde{g}:X_1\cup_f
X_2\to(V_1\cup_{\tilde{f}}V_2)^*\otimes(V_1\cup_{\tilde{f}}V_2)^*$
as
$$\tilde{g}(x)=\left\{
\begin{array}{ll}
((j_1^{-1})^*\otimes(j_1^{-1})^*)\circ g_1(i_1^{-1}(x)) & \mbox{for }x\in i_1(X_1\setminus Y) \\
((j_2^{-1})^*\otimes(j_2^{-1})^*)\circ g_2(i_2^{-1}(x)) & \mbox{for
}x\in i_2(X_2).
\end{array}
\right.$$ Then $\tilde{g}$ is a pseudo-metric on the pseudo-bundle
$\pi_1\cup_{(\tilde{f},f)}\pi_2:V_1\cup_{\tilde{f}}V_2\to X_1\cup_f
X_2$.
\end{thm}

\section{Diffeological vector pseudo-bundles of Clifford algebras}

In this section we consider a finite-dimensional diffeological
vector pseudo-bundle $\pi:V\to X$ carrying a pseudo-metric $g$. As
we have shown above, this cannot be an arbitrary pseudo-bundle,
since there are some that do not admit any pseudo-metric.
Furthermore, as things stand now, it makes sense for us to limit
this discussion to those pseudo-bundles that are obtained by gluing
together some diffeologically trivial pseudo-bundles, usually ones
of differing dimensions and along maps that may or may not be
diffeomorphisms.

\subsection{Defining the pseudo-bundle of Clifford algebras $\cl(V,g)$}

Let $\pi:V\to X$ be a finite-dimensional diffeological pseudo-bundle
endowed with a pseudo-metric $g:X\to V^*\otimes V^*$. Then there is
a rather evident construction of the corresponding pseudo-bundle of
Clifford algebras. We go over it as briefly as possible, since all
the operations involved (direct sum, tensor product, and taking
quotients), and their relevant properties,\footnote{This mainly
means the commutativity between considering the subset diffeology on
each fibre and performing a given operation on the individual fibre
or, respectively, the entire pseudo-bundle.} have already been
described.

\begin{defn}
The \textbf{pseudo-bundle of Clifford algebras
$\pi^{\cl}:\cl(V,g)\to X$} is the space
$$\cl(V,g):=\bigcup_{x\in X}\cl(\pi^{-1}(x),g(x))$$
equipped with the obvious projection $\pi^{\cl}$ to $X$ and endowed
with the finest diffeology such that the subset diffeology on each
fibre $\cl(\pi^{-1}(x),g(x))$ coincides with the diffeology of the
Clifford algebra of the diffeological vector space $\pi^{-1}(x)$
with the pseudo-metric $g(x)$.
\end{defn}

With the definition stated in such terms, we should of course
explain why such diffeology exists. The proof of this coincides with
the explicit construction of the diffeology in question, which is as
follows.

Consider first the \textbf{pseudo-bundle of tensor algebras}
$\pi^{T(V)}:T(V)\to X$, where
$$T(V):=\bigcup_{x\in X}T(\pi^{-1}(x)),$$
with each $T(\pi^{-1}(x)):=\bigoplus_{r}(\pi^{-1}(x))^{\otimes r}$
being the usual tensor algebra of the diffeological vector space
$\pi^{-1}(x)$. The space $T(V)$ is endowed with the vector space
direct sum diffeology relative to the tensor product diffeology on
each factor. By the properties of these diffeologies, the subset
diffeology on each fibre of $T(V)$ is that of the tensor algebra of
the diffeological vector space $\pi^{-1}(x)$.

Next, for each fibre $(\pi^{T(V)})^{-1}(x)=T(\pi^{-1}(x))$ of $T(V)$
let $W_x$ be the kernel of the universal map $T(\pi^{-1}(x))\to
\cl(\pi^{-1}(x),g(x))$. Then, as is generally the case,
$W=\bigcup_{x\in X}W_x\subset T(V)$, with the subset diffeology
relative this inclusion, is a sub-bundle of $T(V)$.

Finally, consider the quotient pseudo-bundle $T(V)/W$. By definition
of the quotient pseudo-bundle, it has $\cl(\pi^{-1}(x),g(x))$ as
fibre at $x$, and so it coincides with $\cl(V,g)$ as defined above.
It then follows from general properties of quotient pseudo-bundles
that, relative to the quotient diffeology, the subset diffeology on
each fibre is the diffeology of the Clifford algebra of the
corresponding fibre $\pi^{-1}(x)$ of $V$.

\subsection{The pseudo-bundle $\cl(V_1\cup_{\tilde{f}}V_2,\tilde{g})$}

We now consider the behavior of pseudo-bundles of Clifford algebras
under the operation of diffeological gluing. We only treat the case
of two factors (which of course may serve as an inductive step for a
longer sequence of gluings, but we do not go into any detail about
this).

\subsubsection{Existence and smoothness of a gluing map for
$\cl(V_1,g_1)$ and $\cl(V_2,g_2)$}

Suppose that we are given a gluing $(\tilde{f},f)$ between the
pseudo-bundles $\pi_1:V_1\to X_1$ and $\pi_2:V_2\to X_2$ (endowed
with a pseudo-metric each). The existence of a natural gluing map
between the corresponding pseudo-bundles of Clifford algebras is
based, more than on anything else, on the universal properties of
Clifford algebras, together with the commutativity properties
considered above. More precisely, we have the following.

\begin{lemma}
Let $\pi_1:V_1\to X_1$ and $\pi_2:V_2\to X_2$ be finite-dimensional
diffeological vector pseudo-bundles, let $(\tilde{f},f)$ be a pair
of smooth maps defining a gluing between them, let $Y\subseteq X_1$
be the domain of definition of $f$, and let $g_1$ and $g_2$ be
compatible pseudo-metrics on $V_1$ and $V_2$ respectively. Then
there exists a well-defined map
$$\tilde{F}^{\cl}:(\pi_1^{\cl})^{-1}(Y)\to(\pi_2^{\cl})^{-1}(f(Y)),$$
that covers $f$ and whose restriction onto each fibre is an algebra
homomorphism.
\end{lemma}

\begin{proof}
The construction of the map $\tilde{F}^{\cl}$ is the obvious one; on
each fibre $(\pi_1^{\cl})^{-1}(y)$ it is given by extending
$\tilde{f}|_{\pi_1^{-1}(y)}$ by linearity and tensor product
multiplicativity. That this gives a well-defined map
$$\cl(\pi_1^{-1}(y),g_1(y))\to\cl(\pi_2^{-1}(f(y)),g_2(f(y)))$$
follows from the compatibility of
the pseudo-metrics with respect to $g_1$ and $g_2$.
\end{proof}

Thus, in order to be able to claim the existence of a well-defined
gluing between $\cl(V_1,g_1)$ and $\cl(V_2,g_2)$, we only need to
establish the following statement.

\begin{prop}\label{gluing:clifford:alg:bundle:possible:prop}
The map $\tilde{F}^{\cl}$ is smooth.
\end{prop}

\begin{proof}
Let $p:U\to(\pi_1^{\cl})^{-1}(Y)$ be a plot; by definition of the
diffeology on any Clifford algebra, it locally lifts to a plot
$\tilde{p}:U\to T(\pi_1^{-1}(Y))$ of $T(\pi_1^{-1}(Y))$. Observe
next that this lift is compatible with the action of
$\tilde{F}^{\cl}$, in the following sense. If $\tilde{F}^{TV}$ is
the lift of $\tilde{F}^{\cl}$ to an algebra homomorphism
$$\tilde{F}^{TV}:T(\pi_1^{-1}(Y))\to T(\pi_2^{-1}(f(Y))),$$
and
$$\Pi_1:T(\pi_1^{-1}(Y))\to\cl(\pi_1^{-1}(Y),g_1|_{Y}),$$
$$\Pi_2:T(\pi_2^{-1}(f(Y)))\to\cl(\pi_2^{-1}(f(Y)),g_2|_{f(Y)})$$ are the
corresponding universal maps, then we have
$$\Pi_2\circ\tilde{F}^{TV}\circ\tilde{p}=\tilde{F}^{\cl}\circ p.$$ In
particular (this is a consequence of the properties of the
pushforward diffeologies), the map $\tilde{F}^{\cl}$ is smooth if
and only if $\tilde{F}^{TV}$ is smooth.

Let us then show that $\tilde{F}^{TV}$ is smooth, that is, that
$\tilde{F}^{TV}\circ\tilde{p}$ is a plot of $T(\pi_2^{-1}(f(Y)))$
(where $\tilde{p}$ can now be any plot of $T(\pi_1^{-1}(Y))$; we do
not have to make a reference to it being a lift of a plot of the
corresponding Clifford algebras' bundle). We can assume that the
domain of definition of $\tilde{p}$ is small enough so that, one,
its range is contained in $\bigoplus_{r=0}^n(\pi_1^{-1}(Y))^{\otimes
r}$ for some $n$, and, two,
$\tilde{p}=\bigoplus_{r=0}^n\tilde{p}_i$, where
$\tilde{p}_i:U\to\underbrace{(\pi_1^{-1}(Y))\otimes\ldots\otimes(\pi_1^{-1}(Y))}_i$
is a plot of a single summand (the tensor product of $i$ copies of
$\pi_1^{-1}(Y)$).

Obviously, the composition $\tilde{F}^{TV}\circ\tilde{p}$ also
splits as the analogous sum of the compositions
$\bigoplus_i\tilde{F}^{TV}\circ\tilde{p}_i$, with range a single
summand $(\pi_2^{-1}(f(Y)))\otimes\ldots\otimes(\pi_2^{-1}(f(Y)))$,
so it suffices to show that each $\tilde{F}^{TV}\circ\tilde{p}_i$ is
a plot of this tensor product. But by definition $\tilde{F}^{TV}$
preserves the tensor product, and each individual factor it
coincides with (the appropriate restriction of) $\tilde{f}$. Thus,
we conclude that each $\tilde{F}^{TV}\circ\tilde{p}_i$
$$\tilde{F}^{TV}\circ\tilde{p}_i=
\left(\underbrace{\tilde{f}\otimes\ldots\otimes\tilde{f}}_i\right)\circ\tilde{p}_i.$$
Since the tensor product of smooth maps is smooth, this composition
is indeed a plot of
$(\pi_2^{-1}(f(Y)))\otimes\ldots\otimes(\pi_2^{-1}(f(Y)))$. This
allows us to conclude that $\tilde{F}^{TV}$ is smooth, and therefore
so is $\tilde{F}^{\cl}$.
\end{proof}

Essentially, what we need Proposition
\ref{gluing:clifford:alg:bundle:possible:prop} for, is to be able to
claim that there is a natural gluing of $\pi_1^{\cl}:\cl(V_1,g_1)\to
X_1$ to $\pi_2^{\cl}:\cl(V_2,g_2)\to X_2$, which is along the maps
$\tilde{F}^{\cl}$ and $f$. Furthermore, still by compatibility
between $g_1$ and $g_2$, one can consider the pseudo-bundle of
Clifford algebras associated to
$(V_1\cup_{\tilde{f}}V_2,\tilde{g})$, that is, the pseudo-bundle
$$(\pi_1\cup_{(\tilde{f},f)}\pi_2)^{\cl}:\cl(V_1\cup_{\tilde{f}}V_2,\tilde{g})
\to X_1\cup_f X_2.$$ We now make a comparison between these two
pseudo-bundles, \emph{i.e.}, between
$$\pi_1^{\cl}\cup_{(\tilde{F}^{\cl},f)}\pi_2^{\cl}:\cl(V_1,g_1)\cup_{\tilde{F}^{\cl}}\cl(V_2,g_2)
\to X_1\cup_f X_2$$ and
$$(\pi_1\cup_{(\tilde{f},f)}\pi_2)^{\cl}:\cl(V_1\cup_{\tilde{f}}V_2,\tilde{g})
\to X_1\cup_f X_2,$$ starting with the simpler case: when
$\tilde{f}$ is a diffeomorphism.

\subsubsection{If $\tilde{f}$ is a diffeomorphism}

Expectedly, in this case the gluing of $\cl(V_1,g_1)$ to
$\cl(V_2,g_2)$ along $(\tilde{F}^{\cl},f)$ yields a pseudo-bundle
which is naturally diffeomorphic to the pseudo-bundle
$\cl(V_1\cup_{\tilde{f}}V_2,\tilde{g})$. This is stated in the
following theorem; notice that on the level of pseudo-metrics this
is guaranteed by the fact that on each individual fibre, $\tilde{g}$
coincides with either $g_1$ or $g_2$, as appropriate.

\begin{thm}\label{gluing:clifford:alg:along:diffeo:thm}
Let $\pi_1:V_1\to X_1$ and $\pi_2:V_2\to X_2$ be two
finite-dimensional diffeological vector pseudo-bundles, let
$(\tilde{f},f)$ be a gluing between them such that $\tilde{f}$ is a
diffeomorphism, and $g_1$ and $g_2$ be pseudo-metrics on $V_1$ and
$V_2$ compatible with this gluing. Let $\pi_i^{\cl}:\cl(V_i,g_i)\to
X_i$ for $i=1,2$ be the pseudo-bundle of Clifford algebras. Then the
above-defined map
$$\tilde{F}^{\cl}:(\pi_1^{\cl})^{-1}(Y)\to(\pi_2^{\cl})^{-1}(f(Y))$$
is such that the following two pseudo-bundles are diffeomorphic:
$$\pi_1^{\cl}\cup_{(\tilde{F}^{\cl},f)}\pi_2^{\cl}:
\cl(V_1,g_1)\cup_{\tilde{F}^{\cl}}\cl(V_2,g_2)\to X_1\cup_f X_2,$$
and
$$(\pi_1\cup_{(\tilde{f},f)}\pi_2)^{\cl}:
\cl(V_1\cup_{\tilde{f}}V_2,\tilde{g})\to X_1\cup_f X_2,$$ where
$\tilde{g}$ is the pseudo-metric on the pseudo-bundle
$$\pi_1\cup_{(\tilde{f},f)}\pi_2:V_1\cup_{\tilde{f}}V_2\to X_1\cup_f
X_2$$ induced by the pseudo-metrics $g_1$ and $g_2$.
\end{thm}

We denote the diffeomorphism whose existence is claim in the
theorem, by
$$\Phi^{\cl}:\cl(V_1,g_1)\cup_{\tilde{F}^{\cl}}\cl(V_2,g_2)\to
\cl(V_1\cup_{\tilde{f}}V_2,\tilde{g}).$$

\begin{proof}
Since by construction $\tilde{F}^{\cl}$ is an algebra homomorphism
on each fibre, it is in particular linear; also by construction, it
is a lift of $f$, so it does make sense to speak of gluing of the
pseudo-bundles $\pi_1^{\cl}:\cl(V_1,g_1)\to X_1$ and
$\pi_2^{\cl}:\cl(V_2,g_2)\to X_2$ along $\tilde{F}^{\cl}$ and $f$.
We thus need to exhibit a diffeomorphism $\Phi^{\cl}$ between
$\cl(V_1,g_1)\cup_{\tilde{F}^{\cl}}\cl(V_2,g_2)$ and
$\cl(V_1\cup_{\tilde{f}}V_2,\tilde{g})$ covering the identity
diffeomorphism on $X_1\cup_f X_2$.

The existence of a diffeomorphism
$$\Phi^{\cl}:\cl(V_1,g_1)\cup_{\tilde{F}^{\cl}}\cl(V_2,g_2)\to
\cl(V_1\cup_{\tilde{f}}V_2,\tilde{g})$$ covering the identity map
$X_1\cup_f X_2\to X_1\cup_f X_2$ can now be deduced from the
construction of $\tilde{g}$, that of the map $\tilde{F}^{\cl}$, as
well as the commutativity between the gluing and the operations of
the direct sum and tensor product (see Section 4.2.2). Specifically,
applying the considerations already made to establish the smoothness
of $\tilde{F}^{\cl}$ and the just-mentioned commutativity, one
immediately obtains a diffeomorphism
$$\Phi^{TV}:T(V_1)\cup_{\tilde{F}^{TV}}T(V_2)\to T(V_1\cup_{\tilde{f}}V_2)$$
that covers the identity map on $X_1\cup_f X_2$.

It then suffices to establish that the pushforward of $\Phi^{TV}$ by
the natural projections
$$\pi_1^{TV}\cup_{(\tilde{F}^{TV},\tilde{F}^{\cl})}\pi_2^{TV}:
T(V_1)\cup_{\tilde{F}^{TV}}T(V_2)\to
\cl(V_1,g_1)\cup_{\tilde{F}^{\cl}}\cl(V_2,g_2)$$ and
$$\pi^{TV}:T(V_1\cup_{\tilde{f}}V_2)\to
\cl(V_1\cup_{\tilde{f}}V_2,\tilde{g})$$ is well-defined and
bijective; the diffeologies on the ranges being pushforward
diffeologies --- by these same projections ---, the smoothness is
then automatic. That it is well-defined follows from the fact that
$\Phi^{TV}$ sends the kernel of
$\pi_1^{TV}\cup_{(\tilde{F}^{TV},\tilde{F}^{\cl})}\pi_2^{TV}$ to the
kernel of $\pi^{TV}$. That it is bijective is the result of the fact
that the map $\Phi^{TV}$ is so (and this is because $\tilde{f}$ is a
diffeomorphism to start with), and we obtain the claim
\end{proof}

Finally, observe that from this theorem, it is quite easy to
extrapolate its extension to the case of more consecutive gluings,
as long as they are all done along diffeomorphisms.

\subsubsection{Non-diffeomorphism $\tilde{f}$}

We now briefly consider what happens if the assumption of
$\tilde{f}$ being a diffeomorphism is dropped. Of course, if
$\tilde{f}$ is not a diffeomorphism then neither is
$\tilde{F}^{\cl}$. However, what we are really interested in is the
relation between $\cl(V_1,g_1)\cup_{\tilde{F}^{\cl}}\cl(V_2,g_2)$
and $\cl(V_1\cup_{\tilde{f}}V_2,\tilde{g})$, where, as we are about
to explain, the main point is the commutativity of the gluing
operation with those of the direct sum and the tensor product of
pseudo-bundles, rather than the specific properties of the gluing
map.

\paragraph{The map $\Phi^{TV}$} As we have already observed, the
properties of the map $\tilde{F}^{\cl}$ are largely defined by those
of its lift $\tilde{F}^{TV}:T(V_1)\to T(V_2)$, and the previous
construction of the diffeomorphism
$$\Phi^{\cl}:\cl(V_1,g_1)\cup_{\tilde{F}^{\cl}}\cl(V_2,g_2)\to
\cl(V_1\cup_{\tilde{f}}V_2,\tilde{g})$$ was by pushing forward the
then-diffeomorphism $\Phi^{TV}:T(V_1)\cup_{\tilde{F}^{TV}}T(V_2)\to
T(V_1\cup_{\tilde{f}}V_2)$.

\begin{lemma}
The map $\Phi^{TV}:T(V_1)\cup_{\tilde{F}^{TV}}T(V_2)\to
T(V_1\cup_{\tilde{f}}V_2)$ is a diffeomorphism.
\end{lemma}

The difference of this statement with respect to what was claimed,
regarding the map $\Phi^{TV}$, in the proof of Theorem
\ref{gluing:clifford:alg:along:diffeo:thm}, is the absence of any
assumption on $\tilde{f}$ (when it was a diffeomorphism, we
essentially took this statement for granted).

\begin{proof}
The claim follows from the commutativity between gluing, and the
operations of direct sum and tensor product. In particular, let us
show that $\Phi^{TV}$ is bijective. In the formulas below, we us the
notation $V_x$ for the fibre at $x$ of any pseudo-bundle $V$.

Consider the $r$-th summand in the definition of, respectively,
$$T(V_1)=\bigoplus_k V_1^{\otimes k}=\bigcup_{x_1\in X_1}\bigoplus_k(V_1^{\otimes k})_{x_1}=
\bigcup_{x_1\in X_1}\bigoplus_k((V_1)_{x_1})^{\otimes k},$$
$$T(V_2)=\bigoplus_k V_2^{\otimes k}=\bigcup_{x_2\in X_2}\bigoplus_k(V_2^{\otimes k})_{x_2}=
\bigcup_{x_2\in X_2}\bigoplus_k((V_2)_{x_2})^{\otimes k},$$ and
$$T(V_1\cup_{\tilde{f}}V_2)=\bigoplus_k
(V_1\cup_{\tilde{f}}V_2)^{\otimes k}=\bigcup_{x\in X_1\cup_f
X_2}\bigoplus_k\left((V_1\cup_{\tilde{f}}V_2)^{\otimes k}\right)_x=
\bigcup_{x\in X_1\cup_f
X_2}\bigoplus_k\left((V_1\cup_{\tilde{f}}V_2)_x\right)^{\otimes
k}.$$

Observe first of all that, by its definition, the map
$\tilde{F}^{TV}$ (in addition to being fibre-to-fibre) sends the
$r$th summand of the first direct sum to the $r$th summand of the
second direct sum,
$$\tilde{F}^{TV}\left(((V_1)_y)^{\otimes r}\right)\subseteq
((V_2)_{f(y)})^{\otimes r}\mbox{ for all }y\in Y.$$ Furthermore,
within each tensor product it acts as the tensor product of $r$
times $\tilde{f}$ with itself. The fibre over $i_2(f(y))$ of
$T(V_1)\cup_{\tilde{F}^{TV}}T(V_2)$ is therefore
$$\left(((V_1)_y)^{\otimes r}\right)\cup_{\underbrace{\tilde{f}\otimes\ldots\otimes\tilde{f}}_r}
\left(((V_2)_{f(y)})^{\otimes r}\right).$$ Since gluing commutes
with tensor product, it is quite easy to see that this is the same
as the corresponding fibre of $T(V_1\cup_{\tilde{f}}V_2)$.

Finally, the fact that $\Phi^{TV}$ is smooth follows from its
construction, and the definitions of the diffeologies involved.
Indeed, by definition of a vector space's direct sum diffeology,
$\Phi^{TV}$ is smooth if and only if its restriction on the
corresponding direct sum factors (\emph{i.e.}, on all individual
tensor products and, for the domain, their gluings) is smooth. That
such restriction is smooth for each $k$ is the essence of the
commutativity between the gluing and the tensor product stated in
Section 4.2.2.
\end{proof}

\paragraph{The map $\Phi^{\cl}$} Having established that $\Phi^{TV}$
being a diffeomorphism does \emph{not} depend on $\tilde{f}$ being,
or not being, one, we now draw the analogous conclusion for the
pseudo-bundles of Clifford algebras. This can now also be done in
the manner completely analogous to the proof of Theorem
\ref{gluing:clifford:alg:along:diffeo:thm}.

In the theorem below we assume that $\pi_1:V_1\to X_1$,
$\pi_2:V_2\to X_2$, $g_1$, and $g_2$ are as in Theorem
\ref{gluing:clifford:alg:along:diffeo:thm}, but the gluing pair
$(\tilde{f},f)$ is arbitrary (assuming that it is suitable for
gluing, of course). Also notice that the fact of the existence of
pseudo-metrics compatible with it does impose significant
restrictions on $\tilde{f}$ (which, on the other hand, may still
have a non-trivial kernel, see
\cite{exterior-algebras-pseudobundles}).

\begin{thm}
There is a diffeomorphism
$\Phi^{\cl}:\cl(V_1,g_1)\cup_{\tilde{F}^{\cl}}\cl(V_2,g_2)\to
\cl(V_1\cup_{\tilde{f}}V_2,\tilde{g})$.
\end{thm}

\begin{proof}
Obviously, $\Phi^{\cl}$ is obtained, as before, by pushing forward
$\Phi^{TV}$ via the natural projections. It is also smooth for all
the same reasons as in the previous case (these reasons are, it is
obtained via a pushforward of a smooth map, and the diffeologies on
its domain and its range are pushforward diffeologies as well).

That $\Phi^{\cl}$ is bijective, follows from examining the structure
of individual fibres in the two spaces. The main argument is that
over a point in $i_2(f(y))$ for $y\in Y$, the fibre in the domain of
$\Phi^{\cl}$ is naturally diffeomorphic to
$\cl(\pi_2^{-1}(f(y)),g_2)$, which follows from the definition of
gluing, whereas, since the fibre over $i_2(f(y))$ in
$V_1\cup_{\tilde{f}}V_2$ is $j_2(\pi_2^{-1}(f(y)))$, the
corresponding fibre in the range of $\Phi^{\cl}$ is equal to
$\cl(\pi_2^{-1}(f(y)),\tilde{g}|_{\pi_2^{-1}(f(y))})=\cl(\pi_2^{-1}(f(y)),g_2)$
(this is by definition of $\tilde{g}$). It remains to notice that
the action of $\Phi^{TV}$ and then that of $\Phi^{\cl}$ corresponds
each to the identity under these diffeomorphisms.
\end{proof}

\section{The pseudo-bundles of Clifford modules}

In this concluding section we consider the behavior of the
pseudo-bundles of Clifford modules under the operation of gluing.
Let $X_1$ and $X_2$ be two diffeological spaces, and let
$f:X_1\supset Y\to X_2$ be a smooth map. Consider two
finite-dimensional diffeological vector pseudo-bundles over them,
$\pi_1:V_1\to X_1$ and $\pi_2:V_2\to X_2$, and a smooth
fibrewise-linear lift $\tilde{f}:\pi_1^{-1}(Y)\to\pi_2^{-1}(f(Y))$
of $f$. Suppose that each of these pseudo-bundles carries a
pseudo-metric, $g_1$ and $g_2$ respectively, and suppose that these
pseudo-metrics are compatible for the gluing defined by $\tilde{f}$
and $f$; consider the corresponding pseudo-bundles of Clifford
algebras,
$$\pi_1^{\cl}:\cl(V_1,g_1)\to X_1\,\,\,\mbox{ and }\,\,\,
\pi_2^{\cl}:\cl(V_2,g_2)\to X_2,$$ and the pseudo-bundle obtained by
gluing
$$\cl(V_1,g_1)\cup_{\tilde{F}^{\cl}}\cl(V_2,g_2)=
\cl(V_1\cup_{\tilde{f}}V_2,\tilde{g})\to X_1\cup_f X_2.$$

Suppose furthermore that we are given two pseudo-bundles of Clifford
modules,
$$\chi_1:E_1\to X_1\,\,\,\mbox{ and }\,\,\,\chi_2:E_2\to X_2,$$
that is, for $i=1,2$ there is a smooth pseudo-bundle map
$$c_i:\cl(V_i,g_i)\to\mathcal{L}(E_i,E_i)$$ that covers the identity on
the bases. Suppose finally that there is a smooth fibrewise linear
map
$$\tilde{f}':\chi_1^{-1}(Y)\to\chi_2^{-1}(f(Y))$$ that covers
$f$. We wish to specify under which conditions
$E_1\cup_{\tilde{f}'}E_2$ is a Clifford module over
$\cl(V_1\cup_{\tilde{f}}V_2,\tilde{g})$, via an action induced by
$c_1$ and $c_2$.

\paragraph{Notation} Below we will deal with more than one pair of gluings
of pseudo-bundles at a time; to be able to distinguish between them,
we now introduce a somewhat more complicated version of the notation
for the inductions $j_1$ and $j_2$. Specifically, we denote
$$j_1^{\cl}:\cl(V_1,g_1)\setminus(\pi_1^{\cl})^{-1}(Y)\hookrightarrow
\cl(V_1,g_1)\cup_{\tilde{F}^{\cl}}\cl(V_2,g_2)$$ and
$$j_2^{\cl}:\cl(V_2,g_2)\hookrightarrow
\cl(V_1,g_1)\cup_{\tilde{F}^{\cl}}\cl(V_2,g_2)$$ the two standard
inductions for the pseudo-bundles
$\cl(V_1,g_1)\cup_{\tilde{F}^{\cl}}\cl(V_2,g_2)$. In contrast, we
denote by
$$j_1^{E_1}:E_1\setminus\chi_1^{-1}(Y)\to
E_1\cup_{\tilde{f}'}E_2$$ and
$$j_2^{E_2}:E_2\hookrightarrow E_1\cup_{\tilde{f}'}E_2$$ the two
standard inductions for the pseudo-bundle $E_1\cup_{\tilde{f}'}E_2$.

\paragraph{Compatibility of $c_1$ and $c_2$} Let $y\in Y$, and let
$v\in(\pi_1^{\cl})^{-1}(y)$. Consider the element
$$\Phi^{\cl}(j_2^{\cl}(\tilde{F}^{\cl}(v)))\in((\pi_1\cup_{(\tilde{f},f)}\pi_2)^{\cl})^{-1}(i_2(f(y)))\subset
\cl(V_1\cup_{\tilde{f}}V_2,\tilde{g});$$ the two
relevant\footnote{In the sense of the two points in the base that
come into play here.} actions are
$$c_1(v):(E_1)_y\to(E_1)_y,\mbox{ and }c_2(\tilde{F}^{\cl}(v)):(E_2)_{f(y)}\to(E_2)_{f(y)}.$$
We shall now compare the following two:
$$\tilde{f}'\circ(c_1(v)),\mbox{ and }(c_2(\tilde{F}^{\cl}(v)))\circ\tilde{f}'.$$

Our aim is to obtain a well-defined action on
$E_1\cup_{\tilde{f}'}E_2$, and a potential issue presents itself
precisely for elements of form
$\Phi^{\cl}(j_2^{\cl}(\tilde{F}^{\cl}(v)))$. Indeed, this element
acts on the fibre $(E_1\cup_{\tilde{f'}}E_2)_{i_2(f(y))}$. Its
action therefore can potentially be described by
$$\Phi^{\cl}(j_2^{\cl}(\tilde{F}^{\cl}(v)))(\tilde{f}'(e_1))=
\tilde{f}'(c_1(v)(e_1))\mbox{ for any }e_1\in(E_1)_y.$$
On the other hand, for a generic element $e_2\in(E_2)_{f(y)}$
(generic meaning that it does not have to belong to the image of
$\tilde{f}'$) we might have
$$\Phi^{\cl}(j_2^{\cl}(\tilde{F}^{\cl}(v)))(e_2)=c_2(\tilde{F}^{\cl}(v))(e_2).$$
In order to obtain a smooth induced action, we wish to ensure that
these expressions are compatible with each other. We thus obtain the
following notion.

\begin{defn}
The actions $c_1$ and $c_2$ are \textbf{compatible} if for all $y\in
Y$, for all $v\in(\pi_1^{\cl})^{-1}(y)$, and for all
$e_1\in\chi_1^{-1}(y)=(E_1)_y$ we have
$$\tilde{f}'(c_1(v)(e_1))=c_2(\tilde{F}^{\cl}(v))(\tilde{f}'(e_1)).$$
\end{defn}

\paragraph{Compatibility of actions as an instance of compatibility of maps}
Let us now see whether the notion of two compatible actions is an
instance of the usual $(f,g)$-compatibility for smooth maps between
diffeological spaces (see \cite{pseudometric-pseudobundle}, Section
4.2). Each $c_i$ can be seen as a pseudo-bundle map
$c_i:\cl(V_i,g_i)\to\mathcal{L}(E_i,E_i)$; on the relevant subset of
$\cl(V_1,g_1)$ (this subset is $\cl(\pi_1^{-1}(Y),g_1|_Y)$) we have
the map $\tilde{F}^{\cl}$. Let us consider the possibility of there
being a lift $\tilde{F}^{mod}$ of $\tilde{F}^{\cl}$ for the
corresponding subsets of $\mathcal{L}(E_1,E_1)$ and
$\mathcal{L}(E_2,E_2)$.

This prospective lift, first of all, should act between the
restricted pseudo-bundles, namely, between $\bigcup_{y\in
Y}L^{\infty}((E_1)_y,(E_1)_y)$ and $\bigcup_{y'\in
f(Y)}L^{\infty}((E_2)_{y'},(E_2)_{y'})$. It should be noticed at
this point that the possibility to define the desired lift from the
actions $c_1$ and $c_2$ is not guaranteed, since the image of $c_1$
may not include whole fibres over $Y$, meaning, more precisely, that
for a given $y\in Y$ the image of the restriction of $c_1$ to
$\cl((V_1)_y,g_1(y))$ is a subset (and a subspace, of course) of
$L^{\infty}((E_1)_y,(E_1)_y)$ (the fibre at $y$ of
$\mathcal{L}(E_1,E_1)$), but it may be a proper subspace. Now, what
we want of
$$\tilde{F}^{mod}:\mathcal{L}(E_1,E_1)\supset(\pi_1^{\mathcal{L}})^{-1}(Y)\to
(\pi_2^{\mathcal{L}})^{-1}(f(Y))\subset\mathcal{L}(E_2,E_2)$$ is
that for every $y\in Y$ and for every $v\in(\pi_1^{\cl})^{-1}(y)$
there be the equality
$$\tilde{F}^{mod}(c_1(v))=c_2(\tilde{F}^{\cl}(v));$$
and if $c_1|_{\cl((V_1)_y,g_1(y))}$ is not surjective, this equality
is not sufficient to determine $\tilde{F}^{mod}$.

\paragraph{The induced action} The above considerations however do
not prevent us from defining an induced action, meaning a
homomorphism
$$c:\cl(V_1\cup_{\tilde{f}}V_2,\tilde{g})\to
\mathcal{L}(E_1\cup_{\tilde{f}'}E_2,E_1\cup_{\tilde{f}'}E_2).$$
Recalling the above identification by the diffeomorphism
$(\Phi^{\cl})^{-1}$
$$\cl(V_1\cup_{\tilde{f}}V_2,\tilde{g})\cong\cl(V_1,g_1)\cup_{\tilde{F}}\cl(V_2,g_2),$$
the action $c$ can be described as follows:
$$c(v)(e)=\left\{\begin{array}{ll}
j_1^{E_1}\left(c_1((j_1^{\cl})^{-1}(v))((j_1^{E_1})^{-1}(e))\right)
& \mbox{if }v\in\mbox{Im}(j_1^{\cl})\Rightarrow
e\in\mbox{Im}(j_1^{E_1}), \\
j_2^{E_2}\left(c_2((j_2^{\cl})^{-1}(v))((j_2^{E_2})^{-1}(e))\right)
& \mbox{if }v\in\mbox{Im}(j_2^{\cl})\Rightarrow
e\in\mbox{Im}(j_2^{E_2}).
\end{array}\right.$$
Since the images of the inductions $j_1^{\cl},j_2^{\cl}$ are
disjoint and cover $\cl(V_1,g_1)\cup_{\tilde{F}}\cl(V_2,g_2)$, and
those of $j_1^{E_1},j_2^{E_2}$ cover $E_1\cup_{\tilde{f}'}E_2$ (and
are disjoint as well), this action is well-defined. Furthermore,
each $c(v)$ is an endomorphism of the corresponding fibre, because
both $c_1((j_1^{\cl})^{-1}(v))$ and $c_2((j_2^{\cl})^{-1}(v))$
(whichever is relevant) are so.

By construction, $c$ is an action of
$\cl(V_1,g_1)\cup_{\tilde{F}}\cl(V_2,g_2)$; to obtain that of the
diffeomorphic $\cl(V_1\cup_{\tilde{f}}V_2,\tilde{g})$, it suffices
to pre-compose $c$ with the diffeomorphism $(\Phi^{\cl})^{-1}$. In
other words, it suffices to substitute $v$ in the above formula with
$(\Phi^{\cl})^{-1}(v')$ for any $v'\in
\cl(V_1\cup_{\tilde{f}}V_2,\tilde{g})$.

\begin{thm}
The action $c$ is smooth as a map
$$\cl(V_1,g_1)\cup_{\tilde{F}}\cl(V_2,g_2)\to\mathcal{L}(E_1\cup_{\tilde{f}'}E_2,E_1\cup_{\tilde{f}'}E_2).$$
\end{thm}

\begin{proof}
Let $p:U\to\cl(V_1,g_1)\cup_{\tilde{F}}\cl(V_2,g_2)$ be a plot; we
need to show that $U\ni u\mapsto c(p(u))$ is a plot of
$\mathcal{L}(E_1\cup_{\tilde{f}'}E_2,E_1\cup_{\tilde{f}'}E_2)$.
Recall that the latter means that for any plot $q:U'\to
E_1\cup_{\tilde{f}'}E_2$ the evaluation map $(u,u')\mapsto
c(p(u))(q(u'))\in E_1\cup_{\tilde{f}'}E_2$ is smooth, for the subset
diffeology of the subset of $U\times U'$ on which it is defined.

Take an arbitrary $q$ and recall that locally it lifts to either a
plot $q_1$ of $E_1$ or a plot $q_2$ of $E_2$. In the first case $q$
has form
$$q(u')=\left\{\begin{array}{ll} j_1^{E_1}(q_1(u')) &
\mbox{if }q_1(u')\in\chi_1^{-1}(X_1\setminus Y),\\
j_2^{E_2}(\tilde{f}'(q_1(u'))) & \mbox{if }q_1(u')\in\chi_1^{-1}(Y).
\end{array}\right.
$$
In the second case $q$ has form $q=j_2^{E_2}\circ q_2$. For each of
these, we need to show that $c(p(u))(q(u'))$ is again a plot of the
same (respective, since $c(\cdot)$ acts fibrewise) form.

Likewise, $p$ is a plot of
$\cl(V_1,g_1)\cup_{\tilde{F}}\cl(V_2,g_2)$; therefore it also lifts
to either a plot $p_1$ of $\cl(V_1,g_1)$ or a plot $p_2$ of
$\cl(V_2,g_2)$. In the former case $p$ has form
$$p(u)=\left\{\begin{array}{ll} j_1^{\cl}(p_1(u)) & \mbox{if }p_1(u)\notin(\pi_1^{\cl})^{-1}(Y),\\
j_2^{\cl}(\tilde{F}^{\cl}(p_1(u))) & \mbox{if
}p_1(u)\in(\pi_1^{\cl})^{-1}(Y).
\end{array}\right.
$$
In the latter case we have $p=j_2^{\cl}\circ p_2$. Notice
furthermore that, if we assume the domain of the evaluation map to
be small enough, then either $p$ lifts to $p_1$ and $q$ lifts to
$q_1$ (but not to $q_2$), or $p$ lifts to $p_2$ and $q$ lifts to
$q_2$.\footnote{This is the only place where we use the fact that
$c$ acts between total spaces of pseudo-bundles.} Thus, there are no
more than two cases that determine the value of the evaluation map.

Let us consider the first case, \emph{i.e.}, where $p$ and $q$ lift
to $p_1$ and $q_1$ respectively. Putting together the definition of
the action $c$ (given prior to the statement of the proposition),
and the descriptions of $p$ and $q$ for this case, we obtain:
$$c(p(u))(q(u'))=\left\{\begin{array}{ll}
j_1^{E_1}\left(c_1(p_1(u))(q_1(u'))\right) & \mbox{if }p_1(u)\notin(\pi_1^{\cl})^{-1}(Y),\\
j_2^{E_2}\left(c_2(\tilde{F}^{\cl}(p_1(u)))(\tilde{f}'(q_1(u')))\right)
& \mbox{if }p_1(u)\in(\pi_1^{\cl})^{-1}(Y).
\end{array}\right.$$
Recall that by the compatibility of the actions $c_1$ and $c_2$ we
have
$$j_2^{E_2}\left(c_2(\tilde{F}^{\cl}(p_1(u)))(\tilde{f}'(q_1(u')))\right)=
j_2^{E_2}\left(\tilde{f}'(c_1(p_1(u))(q_1(u')))\right).$$
Thus, it remains to observe that this is precisely one of the
possible forms for a plot of $E_1\cup_{\tilde{f}'}E_2$ (which is
what we needed to obtain).

Let us now consider the second case, that of $p=j_2^{\cl}\circ p_2$.
Then it follows immediately from our definition of $c$ that
$$c(p(u))(q(u'))=j_2^{E_2}(c_2(p_2(u))(q_2(u'))),$$
and the desired conclusion follows from the smoothness of $c_2$ and
the fact that $j_2^{\cl}$ is an induction. Since all the cases have
thereby been considered, the statement is proven.
\end{proof}

\section{An example}

We now briefly illustrate the constructions of the previous two
sections, by considering the gluing presentation of a certain
pseudo-bundle $\pi:V\to X$ underlying the internal tangent bundle
$T^{dvs}(\hat{X})$ (see \cite{CWtangent}) of the space
$\hat{X}=\matR\vee\matR$. When the latter space is viewed as the
union of the coordinate axes in $\matR^2$ and endowed with the
subset diffeology coming from the standard diffeology on $\matR^2$,
the resulting internal tangent bundle has the same fibrewise
structure, but the diffeology on both the base space is stronger
than that of our $X$, and the diffeology on the total space is
stronger than the one on $V$. This follows, in particular, from an
example in \cite{watts}).

\subsection{The pseudo-bundle $\pi:V\to X$}

We first describe our pseudo-bundle $\pi:V\to X$, immediately in
terms of a gluing construction of it, and endow it with a
pseudo-metric (corresponding to this gluing). We also specify how it
is related to $T^{dvs}(\hat{X})$.

\subsubsection{The gluing presentation of $\pi:V\to X$}

The pseudo-bundle $\pi:V\to X$ is obtained by gluing together three
other pseudo-bundles. The first of these is the pseudo-bundle
$\pi_0:V_0\to X_0$, where $V_0$ is the standard $\matR^2$ and
$X_0=\{x_0\}$ is a one-point set. The second and the third
pseudo-bundles are $\pi_1:V_1\to X_1$ and $\pi_2:V_2\to X_2$, where
both $V_1$ and $V_2$ are the standard $\matR^2$, and both $X_1$ and
$X_2$ are the standard $\matR$; for convenience, we write $\pi_1$ as
the projection onto the first coordinate, $\pi_1(x,y)=x$, and
$\pi_2$ as the projection onto the second coordinate,
$\pi_2(x,y)=y$.

We then define the following two gluings, the first one of $V_1$ to
$V_0$ and the second one of $V_2$ to $V_0$. These gluings are along
the maps $(\tilde{f}_1,f_1)$ and $(\tilde{f}_2,f_2)$, where the maps
$f_1:\{0\}\to\{x_0\}$ and $f_2:\{0\}\to\{x_0\}$ are the obvious
ones, while $\tilde{f}_1$ and $\tilde{f}_2$ are the standard
embeddings of the lines $\{(0,y)\}$ and $\{(x,0)\}$ into $\matR^2$,
that is, $\tilde{f}_1:V_1\ni(0,y)\mapsto(0,y)\in V_0$ and
$\tilde{f}_2:V_2\ni(x,0)\mapsto(x,0)\in V_0$. The pseudo-bundle
$\pi:V\to X$ is obtained by first gluing $V_1$ to $V_0$ along
$(\tilde{f}_1,f_1)$ and by then gluing $V_2$ to
$V_1\cup_{\tilde{f}_1}V_0$ along
$(j_2^{V_0}\circ\tilde{f}_2,i_2\circ f_2)$.

The resulting pseudo-bundle can also be described as follows:
$$\left\{\begin{array}{l}
V=\{(x,y,z,w)\,|\,xy=0,\,x\neq 0\Rightarrow w=0,\,y\neq
0\Rightarrow z=0\}\subset\matR^4,\\
X=\{(x,y,0,0)\,|\,xy=0\}\subset\matR^4,\\
\pi(x,y,w,z)=(x,y,0,0). \end{array}\right.$$ This presentation in
accordance with the analogous presentations of $\pi_0:V_0\to X_0$,
$\pi_1:V_1\to X_1$, and $\pi_2:V_2\to X_2$, which are as follows:
$$\left\{\begin{array}{lll}
V_0=\{(0,0,z,w)\}, & X_0=\{(0,0,0,0)\}, & \pi_0=\pi|_{V_0},\\
V_1=\{(x,0,z,0)\}, & X_1=\{(x,0,0,0)\}, & \pi_1=\pi|_{V_1},\\
V_2=\{(0,y,0,w)\}, & X_2=\{(0,y,0,0)\}, & \pi_2=\pi|_{V_2}.
\end{array}\right.$$ With respect to these presentations, the
gluing maps assume the following form:
$$\left\{\begin{array}{ll}
f_1(0,0,0,0)=(0,0,0,0) & f_2(0,0,0,0)=(0,0,0,0)\\
\tilde{f}_1(0,0,z,0)=(0,0,z,0) & \tilde{f}_2(0,0,0,w)=(0,0,0,w).
\end{array}\right.$$

\subsubsection{A corresponding pseudo-metric $\tilde{g}$}

To give a common description of the four pseudo-metrics that we will
need, any one of those has the following shape. It is defined, first
of all, on an appropriate subset of the set $\{(x,y,0,0)\}$; to each
such point it assigns a symmetric bilinear form on the set
$\{(x,y,z,w)\}$ with $x,y$ fixed and $z,w$. Thus, in resemblance to
the standard case, we will define our pseudo-metrics by quadratic
forms in $dz$ and $dw$, with the following meaning:
$$\left\{\begin{array}{l}
g'(x,y,0,0)=h_{11}(x,y)dz^2+h_{12}(x,y)dzdw+h_{22}(x,y)dw^2,\\
g'(x,y,0,0)((x,y,z_1,w_1),(x,y,z_2w_2))=h_{11}(x,y)z_1z_2+h_{12}(z_1w_2+z_2w_1)+h_{22}w_1w_2,
\end{array}\right.$$
where $h_{11},h_{12},h_{22}:\matR^2\to\matR$ are some ordinary
smooth functions, on which additional conditions are imposed,
according to the specific case we are treating.

\paragraph{The pseudo-metrics $g_0$, $g_1$, and $g_2$} Using the
notation just introduced, we define the following pseudo-metrics on
$V_1$, $V_2$, and $V_0$, denoted respectively by $g_0$, $g_1$, and
$g_2$:
$$\left\{\begin{array}{l}
g_0(0,0,0,0)=dz^2+dw^2,\\
g_1(x,0,0,0)=(x^2+1)dz^2,\\
g_2(0,y,0,0)=(y^2+1)dw^2.
\end{array}\right.$$
In particular, the compatibility of $g_1$ with $g_0$ (recall that it
refers only to the fibre $\{(0,0,z,0)\}$), and that of $g_2$ with
$g_0$, are both evident from these definitions.

\paragraph{The pseudo-metric $\tilde{g}$} We can thus define the
induced metric $\tilde{g}:X\to V^*\otimes V^*$ (it can be described
as $g_2\cup(g_1\cup g_0)$, although we did not quite introduce this
notation; in any case, this is not an instance of gluing of smooth
maps):
$$\tilde{g}(x,y,0,0)=\left\{\begin{array}{ll}
(x^2+1)dz^2 & \mbox{if }y=0,x\neq 0\\
(y^2+1)dw^2 & \mbox{if }x=0,y\neq 0\\
dz^2+dw^2 & \mbox{if }x=y=0.
\end{array}\right.$$

\subsubsection{Relation to $T^{dvs}(\hat{X})$}

There is an obvious identification of our $X$ with $\hat{X}$, of
course, and that extends to individual fibres. That is, for any
$x\in X$ the fibre $\pi^{-1}(x)\subset V$ is the same as the fibre
in $T^{dvs}(\hat{X})$ over the counterpart of $x$; \emph{i.e.}, it
the standard $\matR^2$ over the crossing point of the two lines, and
it is $\matR$ elsewhere. However, as we already mentioned, the
diffeology of $X$ is strictly finer (smaller) than that of
$\hat{X}$, and the same is true of the diffeologies of $V$ and of
$T^{dvs}(\hat{X})$.

\subsection{The Clifford algebra $\cl(V,\tilde{g})$ and gluing}

All our pseudo-bundle involve fibres of dimension at most $2$, so
the corresponding pseudo-bundles of Clifford algebras have dimension
at most $4$. Thus, we will represent them as subsets in $\matR^6$,
where the first four coordinates $x,y,z,w$ correspond to the above
copy of $\matR^4$ that contains the pseudo-bundles themselves, the
$5$-th coordinate $u_0$ corresponds to the scalar parts (elements of
degree $0$) of the Clifford algebras, and the $6$-th coordinate
$u_2$ corresponds to the elements of degree $2$; it is written with
respect to the generator $\frac{\partial}{\partial
z}\otimes\frac{\partial}{\partial w}$.

\paragraph{The Clifford algebra $\cl(V_0,g_0)$} As a set, it is
given by
$$\cl(V_0,g_0)=\{(0,0,z,w,u_0,u_2)\}.$$ The Clifford multiplication,
which is in fact that of the Clifford algebra of the usual $\matR^2$
with the canonical scalar product, is explicitly described by:
\begin{flushleft}
$(0,0,z_1,w_1,u_0',u_2')\cdot_{\cl}(0,0,z_2,w_2,u_0'',u_2'')=$
\end{flushleft}
\begin{flushright}
$=(0,0,u_0'z_2+u_0''z_1+2w_1u_2''-2u_2'w_2,
u_0''w_1+u_0'w_2-2z_1u_2''+2u_2'z_2,u_0'u_0''-2z_1z_2-2w_1w_2-u_2'u_2'',u_0'u_2''+u_0''u_2'+z_1w_2-w_1z_2)$.
\end{flushright}

\paragraph{The Clifford algebras $\cl(V_1,g_1)$ and $\cl(V_2,g_2)$}
Consistently with the above, these two are given by the sets
$$\cl(V_1,g_1)=\{(x,0,z,0,u,0)\}\,\,\,\mbox{ and }\,\,\,
\cl(V_2,g_2)=\{(0,y,0,w,u,0)\};$$ the Clifford multiplication is
defined, respectively, by
$$(x,0,z_1,0,u_0',0)\cdot_{\cl}(x,0,z_2,0,u_0'',0)=
(x,0,u_0''z_1+u_0'z_2,0,u_0'u_0''-2z_1z_2(x^2+1),0),$$
$$(0,y,0,w_1,u_0',0)\cdot_{\cl}(0,y,0,w_2,u_0'',0)=
(0,y,u_0''w_1+u_0'w_2,0,u_0'u_0''-2w_1w_2(y^2+1),0).$$

\paragraph{The resulting Clifford algebra $\cl(V,\tilde{g})$} It is
a reflection of what has been proven in Section 5 regarding the
existence of the induced gluing between (pseudo-bundles of) Clifford
algebras, that we obtain a presentation of $\cl(V,\tilde{g})$ by
simply uniting the above presentations for the factors:
$$\cl(V,\tilde{g})=\{(0,0,z,w,u_0,u_2)\}\cup\{(x,0,z,0,u,0)\}\cup\{(0,y,0,w,u,0)\}$$
and
\begin{flushleft}
$(x,y,z_1,w_1,u_0',u_2')\cdot_{\cl}(x,y,z_2,w_2,u_0'',u_2'')=$
\end{flushleft}
\begin{flushright}
$=\left\{\begin{array}{ll} (0,0,u_0'z_2+u_0''z_1+2w_1u_2''-2u_2'w_2,
u_0''w_1+u_0'w_2-2z_1u_2''+2u_2'z_2, & \\
u_0'u_0''-2z_1z_2-2w_1w_2-u_2'u_2'',u_0'u_2''+u_0''u_2'+z_1w_2-w_1z_2)
& \mbox{if }x=y=0,\\
(x,0,u_0''z_1+u_0'z_2,0,u_0'u_0''-2z_1z_2(x^2+1),0) & \mbox{if
}y=0\\
(0,y,u_0''w_1+u_0'w_2,0,u_0'u_0''-2w_1w_2(y^2+1),0) & \mbox{if }x=0.
\end{array}\right.$
\end{flushright}
The fact that these formulae are consistent at all points of
intersection reflects the commutativity of the Clifford algebra
construction with gluing, established in Theorem 5.6; indeed, all
the gluing maps are restrictions, to appropriate subsets, of the
identity map $\matR^6\to\matR^6$.

\subsection{The exterior algebras $\bigwedge V$ and $\bigwedge V_i$ as Clifford modules}

The four exterior algebras are defined by the same four subsets of
$\matR^6$; the product, being the usual exterior product, has of
course different form.

\paragraph{The exterior algebra $\bigwedge V_0$} This is the set
$$\bigwedge V_0=\{(0,0,z,w,u_0,u_2)\};$$ the exterior product is
given by
\begin{flushleft}
$(0,0,z_1,w_1,u_0',u_2')\wedge(0,0,z_2,w_2,u_0'',u_2'')=$
\end{flushleft}
\begin{flushright}
$=(0,0,u_0'z_2+u_0''z_1,u_0'w_2+u_0''w_1,u_0'u_0'',u_0'u_2''+u_0''u_2'+z_1w_2-z_2w_1)$.
\end{flushright}

\paragraph{The exterior algebras $\bigwedge V_1$ and $\bigwedge
V_2$} These are sets
$$\bigwedge V_1=\{(x,0,z,0,u,0)\}\,\,\,\mbox{ and }\,\,\,
\bigwedge V_2=\{(0,y,0,w,u,0)\};$$ the respective exterior products
are given by
$$(x,0,z_1,0,u_1,0)\wedge(x,0,z_2,0,u_2,0)=(x,0,u_2z_1+u_1z_2,0,u_1u_2,0),$$
$$(0,y,0,w_1,u_1,0)\wedge(0,y,0,w_2,u_2,0)=(0,y,0,u_2w_1+u_1w_2,u_1u_2,0).$$

\paragraph{The exterior algebra $\bigwedge V$} Finally, the exterior
algebra $\bigwedge V$ is also obtained by uniting the three
presentations:
$$\bigwedge V=\{(0,0,z,w,u_0,u_2)\}\cup\{(x,0,z,0,u,0)\}\cup\{(0,y,0,w,u,0)\};$$
the exterior product is given by
\begin{flushleft}
$(x,y,z_1,w_1,u_0',u_2')\wedge(x,y,z_2,w_2,u_0'',u_2'')=$
\end{flushleft}
\begin{flushright}
$=\left\{\begin{array}{ll}
=(x,y,u_0'z_2+u_0''z_1,u_0'w_2+u_0''w_1,u_0'u_0'',u_0'u_2''+u_0''u_2'+z_1w_2-z_2w_1)
& \mbox{if }x=y=0,\\
(x,y,u_2z_1+u_1z_2,0,u_1u_2,0) & \mbox{if }y=0\\
(0,y,0,u_2w_1+u_1w_2,u_1u_2,0) & \mbox{if }x=0.
\end{array}\right.$
\end{flushright} In fact, this is an instance of the similar commutativity
of gluing with the construction of the exterior algebras'
pseudo-bundles (see \cite{exterior-algebras-pseudobundles}); once
again, all formulae are consistent with each other on all
intersection subsets.

\subsection{The gluing of Clifford actions}

Finally, we write down the standard Clifford actions of the above
Clifford algebras on the corresponding exterior algebras.

\paragraph{The action $c_0$ of $\cl(V_0,g_0)$ on $\bigwedge V_0$}
Let us now specify, relative to our presentation of the two
pseudo-bundle, the shape of the standard Clifford action of
$\cl(V_0,g_0)$ on $\bigwedge V_0$. To avoid too complicated
formulae, and following the standard way, we specify it this time on
the degree 1 generators of $\cl(V_0,g_0)$, which are
$$(0,0,1,0,0,0)\,\,\,\mbox{ and }\,\,\,(0,0,0,1,0,0).$$
We obtain
$$c_0(0,0,1,0,0,0)((0,0,z,w,u_0,u_2))=
(0,0,u_0,-u_2,-z,w),$$
$$c_0(0,0,0,1,0,0)((0,0,z,w,u_0,u_2))=
(0,0,u_2,u_0,-w,-z).$$

\paragraph{The actions $c_i$ of $\cl(V_i,g_i)$ on $\bigwedge V_i$ for $i=1,2$}
Analogously, it suffices to specify the action $c_1(x,0,1,0,0,0)$ of
the degree $1$ generator of $\cl(V_1,g_1)$ and the action
$c_2(0,y,0,1,0,0)$ of the degree $1$ generator of $\cl(V_2,g_2)$. We
obtain
$$c_1(x,0,1,0,0,0)((0,0,z,0,u_0,0))=
(x,0,u_0,0,-z(x^2+1),0).$$
$$c_2(0,y,0,1,0,0)((0,0,0,w,u_0,0))=
(0,y,0,u_0,-w(y^2+1),0).$$
The compatibility of the two actions is
reflected by the fact that these expressions coincide with those in
the preceding paragraph over the point $x=y=0$.

\paragraph{The final action $c$ of $\cl(V,g)$ on $\bigwedge V$}
The action $c$ now is obtained by simply concatenating the
expressions for $c_0$, $c_1$, and $c_2$.

\vspace{1cm}

\noindent University of Pisa \\
Department of Mathematics \\
Via F. Buonarroti 1C\\
56127 PISA -- Italy\\
\ \\
ekaterina.pervova@unipi.it\\

\end{document}